\documentclass[final,leqno]{elsarticle}

\usepackage[utf8]{inputenc}
\usepackage{amsmath,amssymb, amsthm}
\usepackage[ruled,vlined,noend]{algorithm2e}
\usepackage{graphicx}\usepackage{booktabs}
\usepackage{bm}\usepackage[font=footnotesize,format=plain,labelfont=bf]{caption}
\usepackage{subcaption}
\usepackage{multirow}
\usepackage{xspace}
\usepackage{tikz}
\usetikzlibrary{calc}
\usetikzlibrary{patterns}
\usepackage{tabularx}
\usepackage{lineno}
\modulolinenumbers[5]
\usepackage{hyperref}
\usepackage{ifthen}
\usepackage[normalem]{ulem}
\usepackage{verbatim}

\newif\ifistechreport\istechreportfalse
\long\def\ifpaper#1{\ifistechreport\else#1\fi}

\newcommand{\Lop}{\mathcal{L}}

\newcommand{\Rop}{\mathcal{R}}

\newcommand{\dRop}{(\mathcal{R}\partial_n)}
\newcommand{\RopHat}{\hat{\mathcal{R}}}

\newcommand{\FEop}{\mathcal{F}}

\newcommand{\FEsolve}{\mathcal{U}}
\newcommand{\FEsolveh}{\mathcal{U}^h}
\newcommand{\applyFEsolve}[2]{\FEsolve \,[#1; \; #2]}
\newcommand{\applyFEsolveh}[2]{\FEsolveh \,[#1; \; #2]}
\newcommand{\applyblockFEsolve}[3]{\tensor{\FEsolve} \,[#1; \, #2; \, #3]}

\newcommand{\IErep}{\mathcal{A}}
\newcommand{\IEreph}{\mathcal{A}^h}
\newcommand{\IErepBdry}{\bar{\mathcal{A}}}

\newcommand{\genIEop}{\left\{CI + \bar{\mathcal{A}}\right\}}

\newcommand{\pd}[2]{\frac{\partial#1}{\partial#2}}

\newcommand{\SnglOp}{\mathcal{S}}
\newcommand{\SnglOpbd}{\bar\SnglOp}
\newcommand{\SpOp}{\mathcal{S}'}
\newcommand{\SpOpbd}{\bar\SpOp}

\newcommand{\DblOp}{\mathcal{D}}
\newcommand{\DblOpbd}{\bar\DblOp}
\newcommand{\DpOp}{\mathcal{D}'}
\newcommand{\DpOpbd}{\bar\DpOp}

\newcommand{\ee}[1]{$\times 10^{#1}$}
\newcommand{\dens}{\gamma}
\newcommand{\nvec}{\hat n}
\newcommand{\arrstretch}[1]{\renewcommand*{\arraystretch}{#1}}

\newcommand{\hvectwo}[2]{[ #1; \; #2 ]}
\newcommand{\vecdens}{\hvectwo{\dens^i}{\dens^e}}
\newcommand{\tensor}[1]{\underline{\underline{#1}}}

\newtheorem{theorem}{Theorem}
\newtheorem{lemma}{Lemma}

\definecolor{Red}{RGB}{118,9,9}
\definecolor{Green}{RGB}{54,123,0}
\definecolor{Blue}{RGB}{11,16,120}
\definecolor{Purple}{RGB}{65,0,120}
\definecolor{LightBlue}{RGB}{81,166,195}
\definecolor{Gold}{RGB}{211,185,60}
\definecolor{DarkGold}{RGB}{118,97,0}
\definecolor{Pink}{RGB}{210,53,164}
\definecolor{Magenta}{RGB}{118,0,81}
\definecolor{Teal}{RGB}{13,112,119}

\newcommand{\extArrowColor}{Blue}

\newcommand{\extHatColor}{Gold}

\newcommand{\intPatternColor}{black}

\newcommand{\intHatColor}{Green}

\newcommand{\intPatternType}{north east lines}
\newcommand{\infnorm}[2]{\|#2\|_{\infty;#1}}
\newcommand{\infnormvol}[1]{\infnorm{\Omega}{#1}}
\newcommand{\infnormbd}[1]{\infnorm{\partial\Omega}{#1}}
\newcommand{\autoinfnorm}[2]{\left\|#2\right\|_{\infty;#1}}

\newcommand{\twonorm}[2]{\|#2\|_{2;#1}}
\newcommand{\twonormvol}[1]{\twonorm{\Omega}{#1}}

\newcommand{\hie}{h_{\text{ie}}}
\newcommand{\hfe}{h_{\text{fe}}}
\newcommand{\lift}{\mathcal{E}}
\newcommand{\lifth}{\mathcal{E}^h}
\newcommand{\pqbx}{p_{\text{QBX}}}
\newenvironment{roomymat}{\left[\arrstretch{1.4}\begin{array}}{\end{array}\right]}

\usepackage[margin=1in]{geometry}

\begin{document}

\begin{frontmatter}

\title{High-order Finite Element--Integral Equation Coupling on Embedded~Meshes
    }

\author[mech]{Natalie N.~Beams}
\ead{natalie.beams@rice.edu}

\author[cs]{Andreas Kl\"ockner\corref{cor1}}
\ead{andreask@illinois.edu}

\author[cs]{Luke N.~Olson}
\ead{lukeo@illinois.edu}

\cortext[cor1]{Corresponding author}
\address[mech]{Department of Mechanical Science and Engineering, University of Illinois at Urbana-Champaign}
\address[cs]{Department of Computer Science, University of Illinois at Urbana-Champaign}

\begin{abstract}
This paper presents a high-order method
for solving an interface problem for the Poisson equation on embedded meshes
through a coupled finite element and integral equation approach.  The method is capable of handling
homogeneous or inhomogeneous jump conditions without modification and retains high-order
convergence close to the embedded interface. We present finite element-integral equation (FE-IE)
formulations for interior, exterior, and interface problems.
The treatments of the exterior and interface problems are new.
The resulting linear systems are solved through an iterative approach
exploiting the second-kind nature of the IE operator combined with algebraic multigrid
preconditioning for the FE part. Assuming smooth continuations of coefficients and
right-hand-side data, we show error analysis supporting high-order
accuracy. Numerical evidence further supports our claims of efficiency and high-order accuracy for smooth
data.
\end{abstract}


  \end{frontmatter}

\section{Introduction}\label{sec:intro}
The focus of this work is on the following model interface problem for
the Poisson equation:
\begin{subequations}\label{eq:interface}
\begin{alignat}{2}
  -\triangle u(x) &=\,  f(x) &\quad& \textnormal{in $\Omega^i \cup \Omega^e$}\\
                          u^i(x) &=\,cu^e(x) + a(x) &\quad& \textnormal{on $\Gamma$}\\
                  \pd{u^i(x)}{n} &=\,  \kappa \pd{u^e(x)}{n} + b(x) &\quad& \textnormal{on $\Gamma$},
\end{alignat}
\end{subequations}
where two bounded domains $\Omega^i,\Omega^e\subset \mathbb R^d$ are separated by an interface
$\Gamma=\bar \Omega^i \cap \bar \Omega^e$ so that $\partial \Omega^i=\Gamma$.
The restriction of $u$ to domain $\Omega^\alpha$ is written
as $u^\alpha$ ($\alpha\in\{i,e\}$). We assume $\Omega^i$ has a smooth boundary.
Example domains are illustrated in
Figure~\ref{fig:interface}.  The forcing function $f$ may be
discontinuous across $\Gamma$, provided smooth extensions are
available in a large enough region across $\Gamma$, as
discussed in more detail in Section~\ref{sec:int-poisson}.  General interface
problems of this kind describe, e.g., steady-state diffusion in
multiple-material domains, and are closely related to problems from multi-phase
low Reynolds number flow, such as viscous
drop deformation and breakup~\cite{StoneReview1994}.  The presented method is usable and much of the related analysis
are valid in two or three dimensions. Numerical experiments in two dimensions support the validity of our claims; experiments in three dimensions are the subject of future investigation.
\begin{figure}[ht!]
\begin{center}
\begin{tikzpicture}[scale=2.5]
  \draw [fill=white] (-1,-1) rectangle (1,1);
  \draw [thick, black, fill=white] plot [smooth cycle] coordinates {(-0.3,-0.5) (0.5,-0.45) (0.2,0.5) (-0.45,0.3) (-0.2,-0.1)};
  \node at (0.1,-0.1) {$\Omega^i$};
  \node at (0.38, 0.52) {$\Gamma$};
  \node at (0.7, -0.8) {$\Omega^e$};
\end{tikzpicture}
\caption{  Two domains, $\Omega^i$ and $\Omega^e$, separated by interface $\Gamma$.
}\label{fig:interface}
\end{center}
\end{figure}
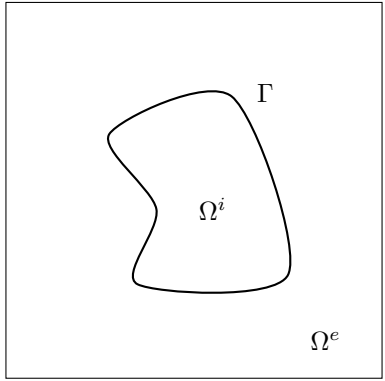

Though finite element methods offer great flexibility with respect to domain geometry, generating domain-conforming meshes is often difficult.  Furthermore, fully unstructured meshes may
require more computation for a given level of uniform accuracy than structured meshes, and
evaluation of the solution at non-mesh points is considerably more complicated than for regular Cartesian
grids.
Consequently, there is much interest in embedded
domain finite element methods, where the problem domain $\Omega$ is placed inside of a larger
domain $\hat{\Omega}$, which may be of any shape but is here chosen to be rectangular and discretized with a structured grid for numerical convenience.
The problem is recast on the new domain while still satisfying the boundary conditions on the original boundary $\partial\Omega$.
Examples of this type of approach include \emph{finite cell methods}~\cite{ParvizianFCM:2007, Kollmannsberger:2015},
which cast the problem in the form of a functional to be minimized, where the functional
is an integral only on the original domain $\Omega$.  These methods
require careful treatment of elements that are partially inside $\Omega$, especially
those containing only small pieces of $\Omega$.  The boundary conditions are enforced weakly using
Lagrange multipliers or penalty terms in the functional.
\emph{Fictitious domain methods}~\cite{Glowinski:1994} avoid special treatment of cut-cell elements by
extending integration outside $\Omega$ and extending the right-hand side as necessary.  This has also been developed
in the context of least-squares finite elements~\cite{ParussiniPediroda:2009}, where
Dirichlet boundary conditions are enforced with Lagrange multipliers.
A conceptual overview of these methods is given in
Figures~\ref{fig:method-fcm} and~\ref{fig:method-fdm}.
\begin{figure}[ht!]
\begin{minipage}[t]{0.3\textwidth}\begin{tikzpicture}[scale=2.5]
  \draw [black, thick, fill=gray!40] plot [smooth cycle] coordinates {(-0.7,-0.5) (0.5,-0.6) (0.6,0.8) (-0.45,0.3) (-0.4,-0.1)};
  \draw [step=0.2,black, very thin] (-1.,-0.8) grid (1.,1.);
  \node at (0.3,-0.1) {$\Omega$};
 \node at (0.7, -0.5) {$\hat\Omega$};
\end{tikzpicture}
\caption{  Finite cell methods use conforming integration within the domain $\Omega$, leading to cut cells.
  The shaded area indicates the domain of integration.
}\label{fig:method-fcm}
\end{minipage}
\hfill
\begin{minipage}[t]{0.3\textwidth}\begin{tikzpicture}[scale=2.5]
  \path [fill=gray!40] (-1., -0.8) rectangle (1.,1.);
  \draw [black, thick, fill=gray!40] plot [smooth cycle] coordinates {(-0.7,-0.5) (0.5,-0.6) (0.6,0.8) (-0.45,0.3) (-0.4,-0.1)};
  \draw [step=0.2,black, very thin] (-1.,-0.8) grid (1.,1.);
  \node at (0.3,-0.1) {$\Omega$};
  \node at (0.7, -0.5) {$\hat\Omega$};
\end{tikzpicture}
\caption{  Fictitious domain methods integrate in the entire introduced domain $\hat\Omega$.
  The shaded area indicates the domain of integration.
}\label{fig:method-fdm}
\end{minipage}
\hfill
\begin{minipage}[t]{0.3\textwidth}\begin{tikzpicture}[scale=2.5]
  \path [use as bounding box] (-1.,-0.8) rectangle (1.,1.);
  \path [pattern=north west lines, pattern color=red] plot [smooth cycle] coordinates
      {(-0.84,-0.6) (0.6,-0.72) (0.72,0.96) (-0.54,0.36) (-0.56,-0.12)};
  \path [fill=white] plot [smooth cycle] coordinates
      {(-0.5,-0.4) (0.40,-0.48) (0.48,0.64) (-0.36,0.24) (-0.28,-0.08)};
 \draw [black, thick] plot [smooth cycle] coordinates {(-0.7,-0.5) (0.5,-0.6) (0.6,0.8) (-0.45,0.3) (-0.4,-0.1)};
  \draw [step=0.2,black, very thin] (-1.,-0.8) grid (1.,1.);
  \node at (0.3,-0.1) {$\Omega$};
  \node at (-0.7, 0.5) {$\hat\Omega$};
\end{tikzpicture}
\caption{
  Immersed interface and immersed finite element methods modify stencils/basis functions near the boundary.
  The area where this might occur is hatched in the figure.
}\label{fig:method-ife}
\end{minipage}
\end{figure}

The \emph{immersed interface method (IIM)}~\cite{LeVequeLiIIM:1994} was introduced to solve elliptic interface problems with
discontinuous coefficients and singular sources using finite differences on a regular Cartesian
grid.  In the IIM, a finite difference stencil is modified to satisfy the interface conditions
at the boundary.  The immersed interface method is closely related is to the \emph{immersed boundary
method}~\cite{MittalIaccarino-Rev:2005}.  The IIM has also been extended to finite element
methods, often called \emph{immersed finite element} methods (IFEM)~\cite{LiIIM-FEM:1998, LiIFEM:2003}.
Similar to the IIM, the IFEM changes the finite element representation by
creating special basis functions that satisfy the interface conditions at
the interface.  The method has also been modified to handle non-homogeneous jump conditions~\cite{GongIFEM:2008, HeIFEM:2011}.  The family of immersed methods is represented by Figure~\ref{fig:method-ife}.

In this paper, we propose a combined finite element-integral equation (FE-IE) method
for solving interface problems such as~\eqref{eq:interface}.  Integral
equation (IE) methods excel at solving homogeneous equations: a solution is
constructed in the entire domain through an equation defined and solved only on
the boundary, resulting in a substantial cost savings over volume-discretizing
methods (including FEM) and reducing the difficulty of mesh generation.
Meanwhile, FE methods deal easily with inhomogeneous equations and other
complications in the PDE\@.  Based on these complementary strengths, the method
presented in this paper combines boundary IE and volume FE methods in a way
that retains the high-order accuracy achievable in both schemes.  As such, the novel contributions of this paper are as follows:
\begin{itemize}
  \setlength\itemsep{0em}
  \item Introduce high-order accurate, coupled FE-IE methods for three foundational problems: interior, exterior, and domains with interfaces;
  \item develop appropriate layer potential representations, leading to integral equations of the second kind;
  \item establish theoretical properties supporting existence and uniqueness of solutions to the various coupled FE-IE problems presented; and
  \item support the accuracy and efficiency of our methods with rigorous numerical tests involving the three foundational problem types.
\end{itemize}
Limitations of the current contribution include the need for smooth
continuation of the right-hand side $f$ across the interface in order
to obtain high-order accuracy, the fact that some of the theory and
our numerical experiments are limited to two dimensions at present
even if the scheme should straightforwardly generalize to three
dimensions in principle, and the need for smooth geometry and
convexity (for some results) to obtain theoretical guarantees of
high-order accuracy.

The literature includes examples of related coupling approaches
combining finite element methods with integral equations for irregular
domains, such as for example
work on the Laplace equation on domains with exclusions~\cite{celorrio_overlapped_2004}
(similar in spirit to the problem of Section~\ref{sec:int-ext}) with a
focus on Schwarz iteration,
work for transmission-type problems on the Helmholtz equation~\cite{dominguez_overlapped_2007},
work on the Stokes equations on a structured mesh~\cite{BirosStokesFE-IE:2003}
or on the Poisson and biharmonic equations~\cite{Mayo:1984}.
These approaches use
a layer potential representation that exists in both the actual domain
$\Omega$ and in $\hat{\Omega}\backslash\Omega$ and is discontinuous across the
domain boundary $\partial\Omega$.  Using known information about the discontinuity in
this representation and the derivatives at the boundary $\partial\Omega$,
modifications to the resulting finite-element stencil are calculated for the differential operator of the PDE
so that the integral representation is valid on the volume mesh.  The coupling
of FE and boundary integral or boundary element methods through a combined
variational problem has been applied to solving unbounded exterior problems,
e.g.\ Poisson~\cite{JohnsonNedelec1980, MeddahiEtAl1996},
Stokes~\cite{SequeiraStokes1983, MeddahiSayasStokes2000}, and wave
scattering~\cite{GaneshMorgenstern2016, HassellSayas2016}.
In contrast, the method of~\cite{RubergCirak:2010}
separates the solution into two additive parts: a finite
element solution found on the regular mesh, and an integral equation solution
defined by the boundary of the actual domain.
Unlike Lagrange multiplier methods or variational coupling, the
boundary conditions in this method are enforced exactly at every
discretization point on the embedded boundary, and no extra variables are
introduced or additional terms added to the finite element functional.
We follow the basic approach of~\cite{RubergCirak:2010} in this
contribution while improving on accuracy, layer potential representations,
and introducing efficient iterative solution approaches.

Unlike the XFEM family of methods~\cite{moes1999finite,belytschko2001arbitrary} which do not support
curvilinear interfaces without further work (e.g.~\cite{cheng2010higher}), the FE-IE method presented here
does not constrain the shape of the interface and does not require the creation of special basis functions to satisfy the boundary conditions; in fact,
the underlying computational implementations of the FE and IE solvers remain largely unchanged.  This is
especially advantageous when considering many different domains $\Omega$, as FE discretizations, matrices, and preconditioners
can be re-used.
The method also avoids the need to impose additional jump conditions to use higher order basis functions, as in~\cite{AdjeridLin:2009}, at the cost of a decreased accuracy for data that cannot be smoothly extended across the boundary.  The method handles
general jump conditions (given by $c$, $\kappa$, $a(x)$ and $b(x)$ in~\eqref{eq:interface})
in a largely unmodified manner.  Indeed, the functions $a(x)$
and $b(x)$ appear only in the right hand side of the problem.  This is achieved through considering a notional splitting of~\eqref{eq:interface}: First, an interior problem embedded in a rectangular fictitious
domain $\hat\Omega$, and second, a domain with an exclusion~---~i.e., identifying the interior domain $\Omega_i$ as the
excluded area~---~as illustrated in Figure~\ref{fig:interface-split}.  Each subproblem is then decomposed into an integral equation
part and a finite element part, with coupling necessary in the case of the domain with an exclusion.  Finally, the two subproblems are coupled through the interface conditions.
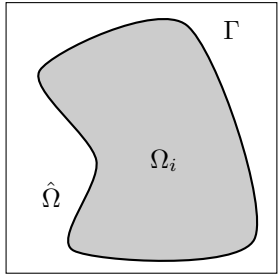
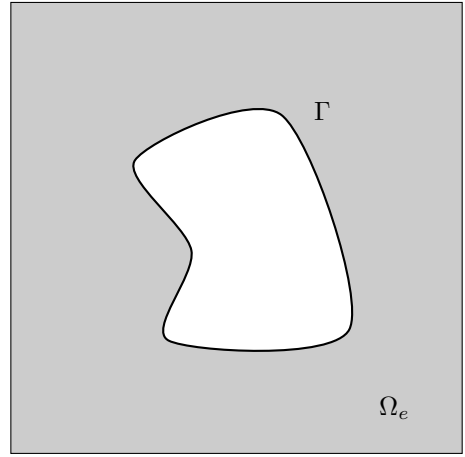
\begin{figure}[ht!]
\begin{center}
\begin{subfigure}{0.45\textwidth}
  \begin{center}
    \begin{tikzpicture}[scale=3]
    \draw [white,fill=white] (-1,-1) rectangle (1,1);
    \draw [fill=white] (-0.6,-0.6) rectangle (0.6,0.6);
    \draw [thick, black, fill=gray!40] plot [smooth cycle] coordinates {(-0.3,-0.5) (0.5,-0.45) (0.2,0.5) (-0.45,0.3) (-0.2,-0.1)};
    \node at (0.1,-0.1) {$\Omega_i$};
    \node at (0.4, 0.48) {$\Gamma$};
    \node at (-0.4, -0.25) {$\hat\Omega$};
    \end{tikzpicture}
  \caption{Interior problem on $\Omega_i$ embedded in a fictitious domain $\hat\Omega$.}\label{fig:interface-int}
  \end{center}
\end{subfigure}\hfill\begin{subfigure}{0.45\textwidth}
  \begin{center}
    \begin{tikzpicture}[scale=3]
    \draw [fill=gray!40] (-1,-1) rectangle (1,1);
    \draw [thick, black, fill=white] plot [smooth cycle] coordinates {(-0.3,-0.5) (0.5,-0.45) (0.2,0.5) (-0.45,0.3) (-0.2,-0.1)};
    \node at (0.38, 0.52) {$\Gamma$};
    \node at (0.7, -0.8) {$\Omega_e$};
    \end{tikzpicture}
  \caption{Problem in $\Omega_e$ is treated as a domain with an exclusion.}\label{fig:interface-int-ext}
  \end{center}
\end{subfigure}
\caption{Splitting of the domains of~\eqref{eq:interface} into
computational subdomains.}\label{fig:interface-split}
\end{center}
\end{figure}

The paper is organized as follows.
First, the FE-IE decomposition is developed for each type of subproblem, starting with the interior embedded domain problem in Section~\ref{sec:int-poisson}.
The form of the error is derived and the method is shown to achieve high-order accuracy.  Then, we present a new splitting for a domain with an exclusion in which the IE part is considered as a pure exterior problem.   This leads to a
a coupled system in Section~\ref{sec:int-ext}. Finally, in Section~\ref{sec:interface}, the interior and exterior subproblems are coupled to solve the interface problem~\eqref{eq:interface}
showing how to retain well-conditioned integral operators and second-kind integral equations in the resulting system of equations.

\subsection{Briefly on Integral Equation Methods}\label{sec:ie-intro}
To fix notation and for the benefit of the reader unfamiliar with
boundary integral equation methods, we briefly summarize the approach
taken by this family of methods when solving boundary value problems
for linear, homogeneous, constant-coefficient partial differential
equations.

Let
$\Gamma \subset \mathbb{R}^d$ ($d=2,3$) be a smooth, rectifiable curve.
For a scalar linear partial differential operator $\Lop$ with
associated free-space Green's function $G(x, x_0)$, the single-layer
and double-layer potential operators on a density function $\dens$ are
defined as
\begin{subequations}\label{eq:layerpots}
\begin{align}
\label{eq:slp-definition}
\SnglOp\dens(x) & = \int_{\Gamma} G(x,x_0) \dens(x_0) \,
dx_0,\\
\DblOp\dens(x) & =  \int_{\Gamma} \left(\nabla_{x_0} G(x,x_0) \cdot \nvec(x_0)\right) \dens(x_0) \, dx_0.
\end{align}
\end{subequations}
Here $\nabla_{x_0}$ denotes the gradient with respect to the variable of integration and $\nvec(x_0)$ is the outward-facing
normal vector. In addition, the normal derivatives
of the layer potentials are denoted
\[
\SpOp\dens(x)=\nvec(x)\cdot \nabla_{x}\SnglOp\dens(x),\quad\text{and}\quad
\DpOp\dens(x)=\nvec(x)\cdot \nabla_{x}\DblOp\dens(x),
\]
respectively. For the Laplacian $\triangle$ in two dimensions, $G(x,x_0) = -{(2 \pi)}^{-1} \log(x - x_0)$.
As $\DblOp\dens(x)$ and $\SpOp\dens(x)$ are discontinuous across the boundary $\partial\Omega$,
we use the notation $\DblOpbd\dens(x)$ and $\SpOpbd\dens(x)$ to denote the principal
value of these operators for target points $x\in\partial\Omega$.
Consider, as a specific example, the exterior Neumann problem in two dimensions for the Laplace
equation:
\[
  \triangle u(x) = 0\quad (x\in \mathbb R^d\setminus \Omega),\quad
  ( \nvec (x) \cdot \nabla u(y)) \to g\quad (x\in \partial\Omega, y\to x_+),\quad
  u(x)\to 0 \quad(x\to\infty),
\]
where $\lim_{y\to x_+}$ denotes a limit
approaching the boundary from the exterior of $\Omega$. By
choosing the integration surface $\Gamma$ in the layer potential as $\partial \Omega$ and representing the solution $u$ in terms of a single layer
potential $u(x)=\mathcal S\dens(x)$ with an unknown density function $\dens$, we obtain that
the Laplace PDE and the far-field boundary condition are satisfied by $u$.
The remaining Neumann boundary condition becomes, by way of the well-known
jump relations for layer potentials (see~\cite{Kress:2014})
the integral equation of the second kind
\[
  -\frac \dens2 + \mathcal{\bar S}'\dens = g.
\]
The boundary $\Gamma$ and density $\dens$ may then be discretized and, using the action of $\mathcal{\bar S}'$ and, e.g.\ an iterative solver, solved
for the unknown density $\dens$. Once $\dens$ is known, the representation
of $u$ in terms of the single-layer
potential~\eqref{eq:slp-definition} may be evaluated anywhere in
$\mathbb R^d\setminus\Omega$ to obtain the sought solution $u$ of the
boundary value problem.

\ifpaper{  In composing a representation of the solution to the
  interface problem out of single- and double-layer potentials, the
  objective is to obtain an integral equation of the second kind~---~i.e., of the form $(I+\mathcal A)\dens=g$, where $\mathcal A$ is a
  compact integral operator. The single- and double-layer
  operators, as operators on the boundary, are indeed compact.
  This typically results in benign conditioning that is independent of mesh size
  and allows the application of a Nyström
  discretization~\cite{Kress:2014}.
}

\section{Interior problems}\label{sec:int-poisson}
We base our discussion in this section on
a coupled finite element-integral equation (FE-IE)
method  presented in~\cite{RubergCirak:2010}, which, as described
there, achieves low-order accuracy. In this paper we
modify the discretization to
achieve high-order accuracy, with analysis and numerical data in
support.  In addition, we quantify the extent to which reduced smoothness in the data results in degraded accuracy.
We derive this combined method and present an error analysis for our implementation
of the method, demonstrating the convergence for problems of varying smoothness.

Assume a smooth domain $\Omega \subset \mathbb{R}^n$ and consider the boundary value problem
\begin{align}
-\triangle u(x) = &\, f \qquad x \in \Omega \nonumber \\
u(x) = & \, g \qquad x \in \partial\Omega.
\label{eq:int-poisson}
\end{align}
We introduce a domain $\hat\Omega$ such that $\Omega \subset \hat\Omega$
with $\partial \hat \Omega \cap \partial \Omega = \emptyset$.
We represent the solution of~\eqref{eq:int-poisson} as
\[
  u(x)= u_1(x)+u_2(x) \mathbf{1}_\Omega(x) \qquad x\in\hat \Omega,\\
\]
where $u_1$ is constructed as a finite element solution
obtained on the artificial larger domain $\hat\Omega$ and $u_2$
represents the integral equation solution defined in $\Omega$.
$\mathbf 1_\Omega(x)$ represents the indicator function that
evaluates to 1 if $x\in \Omega$ and $0$ otherwise.
If necessary, the indicator function may
be evaluated to the same accuracy as $u_2$ as the negative double
layer potential with the unit density~---~i.e.,
$\mathbf{1}_\Omega(x)=-\DblOp_{\partial\Omega} 1(x)$.

This yields two problems:
\begin{align}
\text{[FE]} \quad -\triangle u_1(x) & = f \quad x \in \hat\Omega & \quad \text{[IE]} \quad -\triangle u_2(x) & = 0 \qquad\quad x \in \Omega\nonumber \\
   u_1 & = 0  \quad x \in \partial\hat\Omega & u_2 & = g - u_1 \;\;\, x \in \partial\Omega.
\label{eq:int-split}
\end{align}
Because $u_1$ does not depend on $u_2$, the two problems may be solved with forward substitution.
Furthermore, the integral equation solution $u_2$ solves
the Laplace equation; consequently, the finite element solver alone
handles the right hand side of the original problem~\eqref{eq:int-poisson}.

In~\eqref{eq:int-poisson}, data for $f$ is only available on
$\Omega$. In~\eqref{eq:int-split} however, $f$ is assumed to be
defined on the entirety of the larger domain $\hat \Omega$. In many
situations (e.g., when a global expression for the right-hand side is available),
 this poses no particular problem. If a
natural extension of $f$ from $\Omega$ to $\hat\Omega$
is unavailable, it may be necessary to compute one.
In this case, the degree of smoothness of the resulting right-hand $f$
may become the limiting factor in the convergence of the
overall method.
A simple, linear-time (though non-local) method to obtain such
an extension involves the solution of an
(in this case) exterior Laplace Dirichlet problem, yielding an $f$
of class $C^0$. This may be
efficiently accomplished using layer potentials~\cite{askham2017adaptive}.
The use of a biharmonic problem yields a smoother $f\in C^1$, albeit at
greater cost. Below, we show convergence data for various degrees
of smoothness of $f$ but otherwise leave this issue to future work.
\subsection{Finite element formulation}
First, we solve the FE problem, which, in the form of~\eqref{eq:int-split}, is not coupled to the IE part.
The weak form of~\eqref{eq:int-split} requires finding a
\begin{equation}
  u_1 \in H^1_0(\hat\Omega) \text{ such that}\quad
   \FEop(v)\, u_1   =    \mathcal{M}(v)\,f    \quad \forall  v \in H^1_0(\hat\Omega),
  \label{eq:int-weak-poisson}
\end{equation}
where $\FEop$ and $\mathcal{M}$ are defined for any $v \in H^1_0(\hat\Omega)$ as
\begin{equation*}
  \FEop(v)\, u_1 = \int_{\hat\Omega} \nabla u_1 \cdot \nabla v  \, dV \qquad \text{and} \qquad  \mathcal{M}(v)\,f = \int_{\hat\Omega} fv \, dV.
\end{equation*}
For the discrete FE solution of the continuous weak problem (\ref{eq:int-weak-poisson}),
we consider a Galerkin formulation with $u_1$, $v \in V^h \subset H^1_0(\hat\Omega)$.

\subsection{Layer potential representation}
We represent $u_2 = \DblOp\dens$ in terms of the double layer potential of an unknown density $\dens$.

\ifpaper{The resulting integral equation problem in~\eqref{eq:int-split} is}
\begin{equation}
  \left\{-\frac{1}{2}I +\DblOpbd \right\} \dens = g - u_1\rvert_{\partial\Omega}.
  \label{eq:int-ie}
\end{equation}
This operator is known to have a trivial nullspace and thus we are guaranteed
existence and uniqueness of a density solution $\dens\in C^0(\partial\Omega)$ for $g-u_1\rvert_{\partial\Omega} \in C^0(\partial\Omega)$
by the Fredholm alternative for a sufficiently smooth curve $\partial\Omega$~\cite{Kress:2014}.

For concreteness, we next discuss the Nystr\"om discretization for
this problem type, omitting analogous detail for subsequent problem
types.
We fix a family of composite quadrature rules with weights $w_{h,i}$
and nodes $\xi_{h,i} \subset \partial\Omega$ parametrized by the element size
$h$ so that
\[
  \left|
  \sum_{i=1}^{n_h} w_{h,i,x} K^{\IErep}(x,\xi_{h,i}) \dens(\xi_{h,i})
  -
  \int_{\partial\Omega} K^{\IErep}(x,\xi) \dens(\xi) d\xi
  \right|
  \le C h^q
\]
for a kernel $K^{\IErep}$ associated with an integral operator
$\IErep$ and densities $\dens\in C^q(\partial\Omega)$.
(We use this notation throughout, e.g.\ we use $K^{\DblOp}$ to
denote the kernel of the double layer potential $\DblOp$.)
We let
\[
  \IErep^h \dens^h=
  \sum_{i=1}^{n_h} w_{h,i,x} K^{\IErep}(x,\xi_{h,i}) \dens(\xi_{h,i})
\]
for general layer potential operators $\IErep$.
Using a conventional Nyström discretization, the unknown discretized density
\begin{equation}
  \dens^h={[\dens^h(\xi_{h,1}), \dots, \dens^h(\xi_{h,N_h})]}^T
  \label{eq:discrete-density}
\end{equation}
satisfies the
linear system given by
\begin{equation}
  -\frac 12 \dens^h(\xi_{h,i})
  +
  \sum_{j=1}^{N_h} w_{h,j,\xi_{h,i}}  K^{\DblOp}(\xi_{h,i},\xi_{h,j}) \dens^h(\xi_{h,j})
  =g(\xi_{h,i})-u_1(\xi_{h,i}).
  \label{eq:nystrom-linsys}
\end{equation}
Once $\dens^h$ is known, the solution $u_2$ can be computed as
\begin{equation}
  u_2(x)=\sum_{j=1}^{N_h} w_{h,j,x}  K^{\DblOp}(x,\xi_{h,j}) \dens^h(\xi_{h,j}).
  \label{eq:int-ie-volume-eval}
\end{equation}

We note that, although the density is numerically represented only in terms of
pointwise degrees of freedom, $\dens^h$ can be extended to a
function defined everywhere on $\partial\Omega$ by making use of the fact that it
solves an integral equation of the second kind~\eqref{eq:int-ie},
yielding
\[
  \dens^h(x)
  =
  2\sum_{j=1}^{N_h} w_{h,j,x}  K^{\DblOp}(x,\xi_{h,j}) \dens^h(\xi_{h,j})
  -2(g(x)-u_1(x)),
\]
while, on account of~\eqref{eq:nystrom-linsys}, agreeing with the
prior definition of $\dens^h$.
\subsection{Error analysis}In this section we establish a
decoupling estimate that
allows us to express the error in the overall solution in terms of the
errors achieved by the IE and FE solutions to their
associated sub-problems given in~\eqref{eq:int-split}.  For notational
convenience, we introduce $r_1 = u_1\rvert_{\partial\Omega}$
for the restriction of $u_1$ to the boundary of $\Omega$ and
$r_1^h = u_1^h\rvert_{\partial\Omega}$ for the restriction 
of the approximate solution $u_1^h$ in the finite element
subspace $V^h \subset H^1_0(\hat\Omega)$.
\begin{lemma}[Decoupling Estimate]  \label{lem:int-decoupling}  Suppose $\partial\Omega$ is a sufficiently smooth bounding curve,
  and let
  \[
    u^h(x)= u_1^h(x)+u_2^h(x) \mathbf{1}_\Omega(x) \qquad x\in\hat \Omega,\\
  \]
  where $u_1^h$ solves the variational problem~\eqref{eq:int-weak-poisson}
  and where $u_2^h$ is the potential obtained by solving~\eqref{eq:nystrom-linsys} and computed according to~\eqref{eq:int-ie-volume-eval}.
  Further suppose that the family of discretizations $\DblOpbd^h$ is
  collectively compact and pointwise convergent.
  Then the overall solution error satisfies
  \begin{equation}\label{eq:decoupling}
    \infnormvol{u-u^h} \le
    C \left( \infnormvol{u_1-u_1^h }
    +\infnormvol{\DblOp - \DblOp^h}\infnormbd{\dens}
    +\infnormbd{g-g^h}
    +\infnormbd{r_1-r_1^h}
      \right),
  \end{equation}  for a constant $C$ independent of the mesh size $h$ or other
  discretization parameters, as soon as $h$ is sufficiently small.
  In~\eqref{eq:decoupling}, $\gamma$ refers to the solution of
  the integral equation~\eqref{eq:int-ie}.
\end{lemma}

The purpose of Lemma~\ref{lem:int-decoupling} is to reduce the error
encountered in the coupled problem to a sum of errors of boundary
value problems each solved by a single, uncoupled method, so that
standard FEM and IE error analysis techniques apply to each part.

\begin{proof}
By the triangle inequality,
\begin{equation}
  \infnormvol{u-u^h } \leq \infnormvol {u_1-u_1^h } +  \infnormvol{u_2-u_2^h}.
\label{eq:int-err-total-general}
\end{equation}
First consider $\infnormvol{ u_2-u_2^h}$ in the IE solution, which we bound with
\begin{equation}
   \infnormvol{u_2-u_2^h}
   \le
   \infnormvol{(\DblOp - \DblOp^h) \dens}
   +
   \infnormvol{\DblOp^h (\dens - \dens^h)}.
   \label{eq:int-eval-plus-dens}
\end{equation}
To estimate the second term, we make use of the fact that
$-1/2I+\DblOpbd$ is has no nullspace~\cite{Kress:2014} and is thereby
invertible by the Fredholm Alternative. For a sufficiently small
$h$, and because the $\DblOpbd^h$ is
collectively compact and pointwise convergent, Anselone's
Theorem~\cite{anselone1964approximate} yields invertibility of the
discrete operator
$(-I/2+\DblOpbd^h)$ as well as the estimate
\[
  \infnormbd{\dens - \dens^h} \le 2C' \infnormbd{(\DblOpbd^h-\DblOpbd)\dens} +\infnormbd{(g-r_1)-(g^h-r_1^h)}
\]
where
\[
  C'=
  \frac{1+2\infnormbd{{(I-2\DblOpbd)}^{-1}\DblOpbd^h}}
  {1-4\infnormbd{{(I-2\DblOpbd)}^{-1}(\DblOpbd^h-\DblOpbd)\DblOpbd^h}},
\]
which is bounded independent of discretization parameters
once $h$ (and thus $(\DblOpbd^h-\DblOpbd)$) is small enough.
Using submultiplicativity and gathering terms in~\eqref{eq:int-eval-plus-dens}
yields
\begin{equation*}
  \infnormvol{u_2-u_2^h}
  \le
  (1+2C'\infnormvol{\DblOp^h})\infnormvol{\DblOp - \DblOp^h}\infnormbd{\dens}
  +
  \infnormvol{\DblOp^h}
  \infnormbd{(g-g^h)-(r_1-r_1^h)}.
\end{equation*}
Because $\infnormvol{\DblOp^h}$ is bounded independent of
$h$ by assumption and with some constant $C$, we find
\begin{equation}
  \infnormvol{u_2-u_2^h}
  \le
  \\
  C\big(
  \infnormvol{\DblOp - \DblOp^h}\infnormbd{\dens}
  +\infnormbd{g-g^h}
  +\infnormbd{r_1-r_1^h}
  \big),
\end{equation}
allowing us to bound
$\infnormvol{u_2-u_2^h}$ in terms of the quadrature error $\infnormvol{\DblOp - \DblOp^h}$
of the numerical layer potential operator as well as the interpolation
error $\infnormbd{g-g^h}$ and the FEM evaluation error $\infnormbd{r_1-r_1^h}$. The latter,
along with $\infnormvol {u_1-u_1^h }$ is controlled by
conventional $L^\infty$ FEM error bounds, for example the
contribution~\cite{haverkamp_aussage_1984} (2D) or the recent
contribution~\cite[(5) and Thm. 12]{leykekhman_finite_2016} (3D).
These references provide bounds that are applicable
with minimal smoothness assumptions on $f$ and homogeneous Dirichlet
BCs as in~\eqref{eq:int-split}. They apply generally
on convex polyhedral domains, a setting that is well-adapted to our
intended application (cf. Figure~\ref{fig:interface}).
\end{proof}

Analogous bounds can be derived in Sobolev spaces, specifically the
$H^{1/2}(\partial \Omega)$ norm on the boundary and the $H^1(\Omega)$,
which in turn can be related to $L^2$ the norms included in the
results of the numerical experiments described below.
\subsection{Numerical experiments}\label{sec:int-num-results}The method we develop in this paper is
a combined FE-IE solver that improves on the approach in~\cite{RubergCirak:2010} in a number of
ways. First, we make use of a representation of the solution that
gives rise to an integral equation of the second kind, leading to
improved conditioning and the applicability of the Nystr\"om method.
Second, we use high-order accurate quadrature
for the evaluation of the layer potentials, leading to improved
accuracy. The earlier method~\cite{RubergCirak:2010} uses a coordinate
transformation to remove the singularity and employs adaptive quadrature for points near the singularity.
We make use of quadrature by expansion (QBX)~\cite{KloecknerQBX:2013} using the \texttt{pytential}~\cite{pytential-github} library.
QBX evaluates layer potentials on and off the source surface by
exploiting the smoothness of the potential to be evaluated. It forms
local/Taylor expansions of the potential off the source surface
using the fact that the coefficient integrals are non-singular.
Compared with the classical singularity removal method based on polar coordinates,
QBX is more general in terms of the kernels it can handle, and it
unifies on- and off-surface evaluation of layer potentials.
It is also naturally amenable to acceleration via fast algorithms~\cite{gigaqbx2d}.
The finite element terms are evaluated in standard $Q^n$ spaces using the
\texttt{deal.II} library~\cite{dealII84,
BangerthHartmannKanschat2007}.

We consider three different right-hand sides with varying smoothness: a
constant function, a $C^0$ piecewise bilinear function, and a piecewise constant
function.
\begin{align}
  f_\text{c}(x,y) &= 1,
  \label{eq:rhs0}\\
  f_\text{bl}(x,y) &= \xi(x)\xi(y),
   \, \text{where} \,\,
   \xi(z) =
      \begin{cases}
       \phantom{-}\frac{5}{3}z + 1 & \text{if} \quad z \leq 0, \\
        -\frac{5}{3}z + 1 & \text{if} \quad z > 0,\\
     \end{cases}
     \label{eq:rhs1}\\
   f_\text{pw}(x,y) &= \eta(x)\eta(y),
   \, \text{where} \,\,
  \eta(z) =
      \begin{cases}
        -1 & \text{if} \quad z \leq 0, \\
        \phantom{-}1 & \text{if} \quad z > 0. \\
     \end{cases}
    \label{eq:rhs2}
 \end{align}
These cases are selected to expose different levels of
regularity in the problem.  A classical $C^2$ solution is expected
from $f_\text{c}$, while $f_\text{bl}$ and $f_\text{pw}$ admit
solutions only in $H^3$ and $H^2$, respectively.

All problems are defined on the domain
$\Omega = \{x: |x|_2 \leq 0.5\}$, a circle of radius $0.5$.
In addition, the domain is embedded in a square domain $\hat\Omega = [-0.6, -0.6] \times [0.6, 0.6]$, as illustrated in Figure~\ref{fig:int-solution} along with
the solution obtained when using the right-hand side $f_\text{c}$.
\begin{figure}[ht!]
\begin{center}
\begin{subfigure}{0.3\textwidth}
\begin{center}
  \begin{tikzpicture}[scale=1.5]
        \path [fill=white,thick](-1.5, -2) rectangle (1.5,2);
    \draw [fill=\intHatColor!40,thick](-1.2,-1.2) rectangle (1.2,1.2);
    \draw[pattern=\intPatternType, thick, pattern color=\intPatternColor] (0., 0.) circle [radius=1.0];
     \node at (0,0) {$\Omega$};
     \node at (1.0, -0.8) {$\hat\Omega$};
  \end{tikzpicture}
  \caption{FE domain $\hat\Omega =[-0.6,-0.6]\times[0.6,0.6]\;$  surrounding the true domain $\Omega$, where $u_2$ is defined.}
\end{center}
\end{subfigure}\hfill\begin{subfigure}{0.65\textwidth}
\begin{center}
\includegraphics[width=0.95\textwidth]{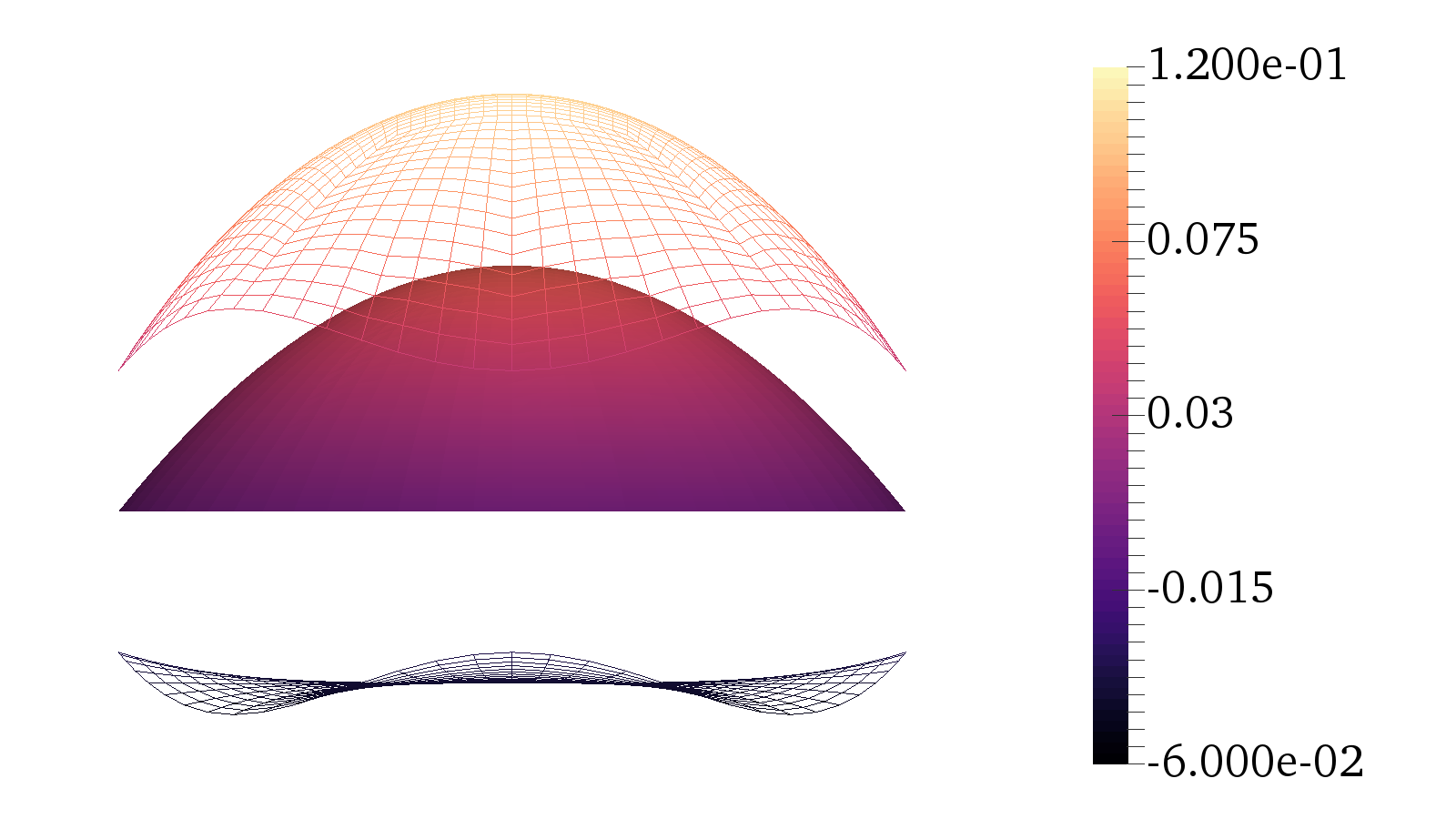}
\caption{FE solution (top wireframe), IE solution (bottom wireframe), and total solution (solid) with right hand side~\eqref{eq:rhs0}.}
\end{center}
\end{subfigure}
\caption{  Solution domains and sample solution for an interior embedded mesh calculation.
}\label{fig:int-solution}
\end{center}
\end{figure}
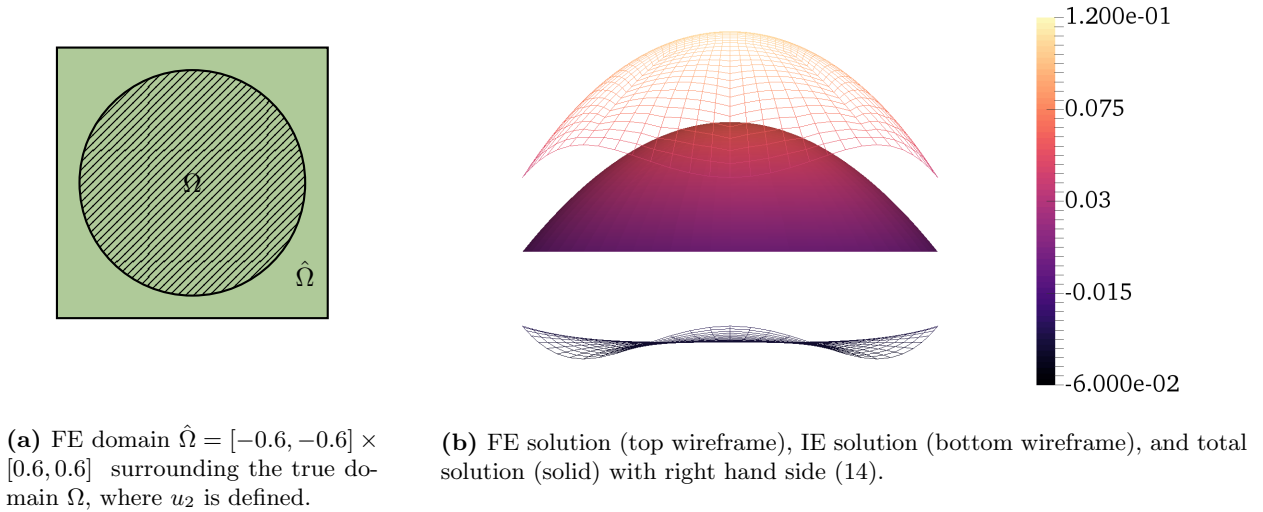

Table~\ref{tab:rhsconv} reports the self-convergence error in the finite element and integral equation portions of the solution for each test case compared to a fine-grid
solution whose parameters are given.  We see that the method exhibits the expected order of accuracy given the smoothness of the data.
In particular, the method is high-order \emph{even near the embedded boundary}, in contrast with standard
immersed boundary methods.
The degree of the FE polynomials space $p$ and the truncation order
$\pqbx$ in Quadrature by Expansion are chosen so as to yield
equivalent orders of accuracy in the solution--for instance $p=1$ and
$\pqbx=2$ yield second-order accurate approximations~\cite{brenner_mathematical_2013,Epstein:2013}.
In the lower-smoothness test cases, we note a marked difference
between the error in the $\infty$- and the 2-norm of the error, both
shown.  Also, for linear basis functions, the finite element
convergence rate in $\infnormvol{\cdot}$  is suboptimal.
This matches analytical expectations as known error estimates of the
error in this norm for this basis include a factor of $\log(1/h)$~\cite{Scott:1976,
Schatz:1980}.
\begin{table}[ht!]
\begin{center}
\caption{Self-convergence to a fine mesh solution $u_{\text{ref}}$ vs.\ smoothness of the right-hand side for the interior problems of Section~\ref{sec:int-num-results}.  EOC refers to the empirical order of convergence.}
\label{tab:rhsconv}
\begingroup
\def\arraystretch{1.2}
\begin{tabular}{ c  c c  c c c c c }
 \multicolumn{3}{c}{\emph{Case}}  & \multicolumn{5}{c}{\emph{Error / Convergence}}  \\
   RHS func.      &       $p$          &    $\pqbx$    &$\hfe$, $\hie$& $\infnormvol{u-u_{\text{ref}} }$& EOC  & $\twonormvol{u-u_{\text{ref}} }$& EOC \\
\toprule
                                     & \multirow{3}{*}{1} & \multirow{3}{*}{2} & 0.04, 0.105    &  4.564\ee{-4}  & --  &  9.576\ee{-5} & --  \\
                                    &                    &                    & 0.02, 0.052    &  9.333\ee{-5}  & 2.3 &  1.975\ee{-5} & 2.3 \\
                                    &                    &                    & 0.01, 0.026    &  1.812\ee{-5}  & 2.4 &  3.982\ee{-6} & 2.3 \\
\cline{2-8}
\multirow{2}{*}{$f_\text{c}$}        & \multirow{3}{*}{2} & \multirow{3}{*}{3} & 0.04, 0.105    &  9.383\ee{-5}  & --  &  2.189\ee{-5} & --  \\
                                   &                    &                    & 0.02, 0.052    &  1.032\ee{-5}  & 3.2 &  2.414\ee{-6} & 3.2 \\
       &                    &                    & 0.01, 0.026    &  8.840\ee{-7}  & 3.5 &  2.052\ee{-7} & 3.6 \\
\cline{2-8}
                                    & \multirow{3}{*}{3} & \multirow{3}{*}{4} & 0.04, 0.105    &  2.405\ee{-5}  & --  &  6.072\ee{-6} & --  \\
                                    &                    &                    & 0.02, 0.052    &  1.510\ee{-6}  & 4.9 &  3.727\ee{-7} & 4.0 \\
                                    &                    &                    & 0.01, 0.026    &  6.887\ee{-8}  & 4.5 &  1.672\ee{-8} & 4.5 \\
\midrule
                                    & \multirow{3}{*}{1} & \multirow{3}{*}{2} & 0.04, 0.105    &  1.091\ee{-4}  &  -- &  3.040\ee{-5} & --  \\
                                    &                    &                    & 0.02, 0.052    &  2.110\ee{-5}  & 2.4 &  6.377\ee{-6} & 2.3 \\
                                    &                    &                    & 0.01, 0.026    &  4.201\ee{-6}  & 2.3 &  1.555\ee{-6} & 2.0 \\
\cline{2-8}
\multirow{3}{*}{$f_\text{bl}$}                  & \multirow{3}{*}{2} & \multirow{3}{*}{3} & 0.04, 0.105    &  2.294\ee{-5}  &  -- &   5.736\ee{-6} & --  \\
                &                    &                    & 0.02, 0.052    &  2.453\ee{-6}  & 3.2 &   6.375\ee{-7} & 3.2 \\
                       &                    &                    & 0.01, 0.026    &  2.219\ee{-7}  & 3.5 &   5.424\ee{-8} & 3.6 \\
\cline{2-8}
      & \multirow{3}{*}{3} & \multirow{3}{*}{4} & 0.04, 0.105    &  5.594\ee{-6}  &  -- &   1.597\ee{-6} & --  \\
                                   &                    &                    & 0.02, 0.052    &  4.103\ee{-7}  & 3.8 &   9.799\ee{-8} & 4.0 \\
                                   &                    &                    & 0.01, 0.026    &  2.396\ee{-8}  & 4.1 &   4.417\ee{-9} & 4.5 \\
\midrule

                                    & \multirow{3}{*}{1} & \multirow{3}{*}{2} & 0.04, 0.105    &  4.918\ee{-4}  &  -- &   1.008\ee{-4} & --  \\
                                    &                    &                    & 0.02, 0.052    &  1.882\ee{-4}  & 1.4 &   2.475\ee{-5} & 2.0 \\
                                    &                    &                    & 0.01, 0.026    &  5.620\ee{-5}  & 1.7 &   4.222\ee{-6} & 2.6 \\
\cline{2-8}
\multirow{3}{*}{$f_\text{pw}$}                  & \multirow{3}{*}{2} & \multirow{3}{*}{3} & 0.04, 0.105    &  2.811\ee{-4}  &  -- &   2.417\ee{-5} & --  \\
                   &                    &                    & 0.02, 0.052    &  9.098\ee{-5}  & 1.6 &   3.583\ee{-6} & 2.8 \\
     &                    &                    & 0.01, 0.026    &  2.489\ee{-5}  & 1.9 &   4.847\ee{-7} & 2.9 \\
\cline{2-8}
                                    & \multirow{3}{*}{3} & \multirow{3}{*}{4} & 0.04, 0.105    &  1.775\ee{-4}  &  -- &   9.453\ee{-6} & --  \\
                                    &                    &                    & 0.02, 0.052    &  4.730\ee{-5}  & 1.9 &   1.294\ee{-6} & 2.9 \\
                                    &                    &                    & 0.01, 0.026    &  1.219\ee{-5}  & 2.0 &   1.594\ee{-7} & 3.0 \\
\bottomrule
 \textit{--- all ---}                         &      4          & 5                  &  0.005, 0.013 &  \multicolumn{4}{c}{\textit{--- reference solution parameters ---}} \\
\bottomrule
\end{tabular}
\endgroup
\end{center}
\end{table}

\section{FE-IE for domains with exclusions}\label{sec:int-ext}As the next step toward solving the interface problem (\ref{eq:interface}), we extend our FE-IE method to a domain with an exclusion as shown in Figure~\ref{fig:interface-split}.
In contrast to the \textit{interior} Poisson problem,
the solution is sought on the intersection of the
unbounded domain $\mathbb{R}^2\setminus \Omega$ and the bounded domain
$\hat\Omega$.  That is,
\begin{align}
\label{eq:int-ext-bvp}
 -\triangle u(x) = &\, f(x) \qquad x \in \hat\Omega\backslash\Omega \nonumber, \\
                    u(x) = &\, g(x) \qquad x \in \partial\Omega \nonumber, \\
                    u(x) = &\, \hat g(x) \qquad x \in \partial\hat\Omega.
\end{align}
The generalization to other boundary conditions is left to future work.

Our new approach to
FE-IE decomposition for this problem is to solve an interior finite element problem on $\hat\Omega$ and an exterior integral equation problem on $\Omega$,
with the two solutions coupled only through their boundary values.  The setup for this problem is illustrated symbolically  in Figure~\ref{fig:int-ext-domain}.
\begin{figure}[!ht]
\begin{subfigure}{0.3\textwidth}
\begin{center}
\resizebox{0.8\textwidth}{!}{
  \begin{tikzpicture}
      \path [fill=white] (-2,-2) rectangle (2,2);
    \draw [fill=gray!40,ultra thick] (-1,-1) rectangle (1,1);
    \draw [fill=white, ultra thick] (0., 0.) circle [radius=0.5];
  \end{tikzpicture}}
  \caption{Problem domain $\hat\Omega\backslash\Omega$}
\end{center}
\end{subfigure}\hfill\begin{subfigure}{0.3\textwidth}
\begin{center}
\resizebox{0.8\textwidth}{!}{
  \begin{tikzpicture}
      \draw[->, ultra thick, \extArrowColor] (0,0) -- (2, 0);
    \draw[->, ultra thick, \extArrowColor] (0,0) -- (2*0.866, 1);
    \draw[->, ultra thick, \extArrowColor] (0,0) -- (1, 2*0.866);
    \draw[->, ultra thick, \extArrowColor] (0,0) -- (0, 2);
    \draw[->, ultra thick, \extArrowColor] (0,0) -- (-2*0.866, 1);
    \draw[->, ultra thick, \extArrowColor] (0,0) -- (-1, 2*0.866);
    \draw[->, ultra thick, \extArrowColor] (0,0) -- (-2, 0);
    \draw[->, ultra thick, \extArrowColor] (0,0) -- (-2*0.866, -1);
    \draw[->, ultra thick, \extArrowColor] (0,0) -- (-1, -2*0.866);
    \draw[->, ultra thick, \extArrowColor] (0,0) -- (0, -2);
    \draw[->, ultra thick, \extArrowColor] (0,0) -- (2*0.866, -1);
    \draw[->, ultra thick, \extArrowColor] (0,0) -- (1, -2*0.866);
     \filldraw[white, even odd rule,inner color=\extArrowColor,outer color=white] (0, 0) circle [radius=1.7];
    \draw[fill=white, ultra thick] (0., 0.) circle [radius=0.5];
    \node at (0,0) {$\Omega$};
    \node at (0.8, -0.22) {$\partial\Omega$};
  \end{tikzpicture}}
  \caption{IE domain $\mathbb{R}^d\backslash\Omega$}
\end{center}
\end{subfigure}\hfill\begin{subfigure}{0.3\textwidth}
\begin{center}
\resizebox{0.8\textwidth}{!}{
  \begin{tikzpicture}
      \path [fill=white] (-2,-2) rectangle (2,2);
    \draw [fill=\extHatColor!60,ultra thick](-1,-1) rectangle (1,1);
    \node at (1.3, -0.7) {$\partial\hat\Omega$};
    \node at (0, 0) {$\hat\Omega$};
  \end{tikzpicture}}
  \caption{FE domain $\hat\Omega$}
\end{center}
\end{subfigure}\caption{Examples of an actual problem domain and the two computational domains for the
coupled FE-IE problem of Section~\ref{sec:int-ext} on a domain with a circular hole.}
\label{fig:int-ext-domain}
\end{figure}
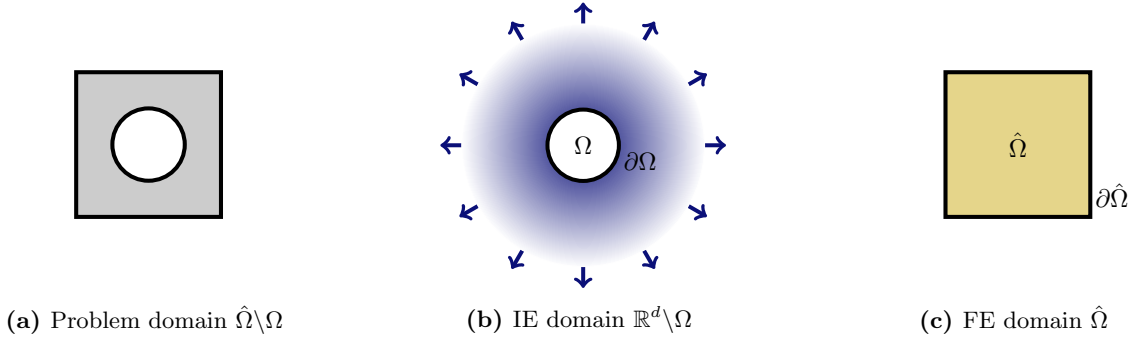
In this way, we allow both methods to play to their individual strengths: the finite element solution exists on a regular, bounded mesh with no
exclusions, while the layer potential solution handles boundary
conditions present on the
boundary of an exclusion, $\partial \Omega$, that is potentially geometrically complex.
\subsection{FE-IE decomposition}\label{sec:FE-IE-decomp}We solve~\eqref{eq:int-ext-bvp} as before by representing
\[
  u(x)=u_1(x)+u_2(x)\quad(x\in\hat\Omega \setminus\Omega)
\]
and posing a system of BVPs for $u_1$ and $u_2$:
\begin{align}
  \label{eq:int-ext}
  \text{[FE]} \quad -\triangle u_1(x) & = f(x), & \quad x \in \hat\Omega, \\
  u_1(x) & = \hat g(x) - u_2(x), &  x \in \partial\hat\Omega, \nonumber\\
  \text{[IE]} \quad -\triangle u_2(x) & = 0,  &x \in \mathbb{R}^d\backslash\Omega,\nonumber\\
  u_2(x) & = g(x) - u_1(x), & x \in \partial\Omega,\nonumber\\
  u_2(x) - A\log|x| & = o(1) & x\to\infty\nonumber \quad(d=2),\\
  u_2(x) & = o(1) & x\to\infty\nonumber \quad(d=3),
\end{align}
with a given constant $A$. In two dimensions,~\eqref{eq:int-ext} includes a far-field boundary
condition for $u_2$ that differs from the standard far-field BC for the
exterior Dirichlet problem, $u(x)=O(1)$ as
$x\to\infty$~\cite{Kress:2014}. There are two reasons for this modification.
First, the BVP~\eqref{eq:int-ext-bvp} allows solutions
containing a logarithmic singularity within $\Omega$. Without
permitting logarithmic blow-up at infinity, such solutions would be
ruled out by the splitting~\eqref{eq:int-ext}.
Neither $u_1$ (nonsingular throughout $\Omega$) nor $u_2$ would be able
to represent them. Second, allowing nonzero additive constants in $u_2$
would lead to a non-uniqueness, since for any given constant $C$,
$(u_1^\text{new}, u_2^\text{new})=(u_1+C,u_2-C)$ would likewise be an
admissible solution.

Next, we support the existence of a solution of the coupled BVPs~\eqref{eq:int-ext} with
the stricter-than-conventional decay condition $o(1)$, we let $A=0$ without loss of generality.
We remind the reader that any solution of the exterior Dirichlet
problem in two dimensions may be represented as $\DblOp\dens +
C$, for some constant $C$~\cite[Thm.~6.25]{Kress:2014}.  Since
$(\DblOp\dens)(x)=O(|x|^{-1})$ ($x\to\infty$)~\cite[(6.19)]{Kress:2014}, the only
loss from our more restrictive decay condition is a constant,
which, as discussed above, may be contributed by $u_1$.

From~\eqref{eq:int-ext}, we see that the two subproblems are now
fully coupled. We cast the subproblem for $u_1$ in variational form in anticipation of
FEM discretization and the subproblem for $u_2$ in terms of layer
potentials. To arrive at the coupled system, we first define
the operator $\Rop$ as the restriction to the boundary $\partial\Omega$ and $\RopHat$ as the restriction to the boundary $\partial\hat\Omega$.

We write $u_2$ in terms of an unknown density $\dens$ using
a layer potential operator $\IErep$ such that $u_2 = \IErep\dens$,
while ensuring that the resulting integral equation is of the second kind:
\begin{equation}
\genIEop \dens = g - \Rop u_1,
\end{equation}
for some constant $C$.

Next, we decompose $u_1$ as
\begin{equation}
 u_1 = \tilde{u}_1 + \hat u_1 =\tilde u_1 +\lift \hat r_1,
\end{equation}
where $\tilde u_1 \in H^1_0(\hat\Omega)$ is zero on the boundary $\partial\hat\Omega$ and
$\hat u_1 \in H^1(\hat\Omega)$ is used to enforce the boundary conditions. $\hat u_1$ is defined by a lifting operator $\lift: H^{1/2}(\partial\hat\Omega) \rightarrow H^1(\hat\Omega)$ that selects a specific $\hat u_1$ in the volume from its boundary restriction $\hat r_1 \equiv \RopHat\hat{u}_1$.
(The precise choice of the lifting operator within these guidelines
has no influence on the obtained solution $\hat u_1$.)

The coupled problem is then to find  $\dens \in C(\partial\Omega)$, $\tilde u_1 \in H^1_0(\hat\Omega)$, and $\hat r_1 \in H^{1/2}(\partial\hat\Omega)$  such that
\begin{equation}
  \begin{roomymat}{ccc}
    \ CI + \bar\IErep    &  \Rop         & \Rop\lift          \\
        0                        & \FEop(v)  &  \FEop(v) \lift  \\
    \RopHat\IErep &    0  &I
  \end{roomymat}
  \begin{roomymat}{c}
    \dens \\
    \tilde u_1 \\
    \hat r_1
  \end{roomymat}
  =
  \begin{roomymat}{c}
    g \\
    \mathcal{M}(v)\, f \\
    \hat g
  \end{roomymat}
  \qquad \forall v \in H^1_0(\hat\Omega),
  \label{eq:int-ext-coupled}
\end{equation}
where $\IErep \dens$ for $u_2$ is used in~\eqref{eq:int-ext}.

Next, we isolate the density equation in~\eqref{eq:int-ext-coupled} using a Schur complement, which results in
\begin{equation}\label{eq:densitywithFE}
  \left\{C I + \bar\IErep -
  \Rop \FEsolve
  \begin{roomymat}{c}
    0\\
    \RopHat \IErep
  \end{roomymat}\right\} \dens
  =
  g -
  \Rop \FEsolve
  \begin{roomymat}{c}
    f \\
    \hat g
  \end{roomymat},
\end{equation}
with the \emph{solution operator} $\FEsolve  : L^2(\hat\Omega) \times
H^{1/2}(\partial\hat\Omega) \rightarrow H^1(\hat\Omega)$, where
$\applyFEsolve{\zeta}{ \hat \rho}$ is defined
as the function $\mu = \tilde{\mu}+\lift \hat\rho$, and where $\tilde{\mu}\in H^1_0(\hat \Omega)$ satisfies
\begin{equation}
  \FEop(v) (\tilde\mu + \lift \hat \rho) = \mathcal M (v) \zeta,
  \quad v \in H^1_0(\hat \Omega).
  \label{eq:fe-solve-op-variational}
\end{equation}
This allows us to express the IE solution to the coupled
problem~\eqref{eq:int-ext-coupled} in terms of input data $f$, $g$,
and $\hat g$, along with the action of $\FEsolve$.   Once $\dens$ is
known, the FE solution is found as
$u_1 = \applyFEsolve{ f }{\hat g - \RopHat\IErep \gamma}$.

The form in~\eqref{eq:densitywithFE} identifies two remaining issues.  The first is the choice of $C$, which
is determined by selecting a layer potential
representation $\IErep$ of $u_2$.  The conventional choice, a double layer potential, is not suitable because the exterior limit of the double
layer operator $\DblOp$ has a nullspace spanned by the
constants.  A common way of addressing this issue involves adding a
layer potential with a lower-order singularity~\citep{Kress:2014};
however, this is inadequate for our coupled FE-IE system (for
$d=2$), as we explain below.  Instead, we choose $\IErep = \DblOp + \SnglOp$, the
sum of the double and single layer potentials, each with the same
density. This choice, by the exterior jump relations for the single
and double layer potentials \citep{Kress:2014} establishes $C=1/2$.

Second, uniqueness and existence for~\eqref{eq:int-ext-coupled} hinges
on joint compactness of the composition of operators $\Rop\applyFEsolve{0}{\RopHat\IErep}$
using the next lemma.
\begin{lemma}  \label{lem:compactcoupling}
  Let $\hat\Omega\subseteq \mathbb R^n$ ($n=2,3$) be bounded, satisfy
  an exterior sphere condition at every boundary point,
  and contain a domain $\Omega$ with a $C^\infty$ boundary. Further assume that
  $d(\partial\hat\Omega, \partial \Omega)> 0$.
  Then the operator
  $\Rop\applyFEsolve{0}{\RopHat\IErep}: C(\partial\Omega) \rightarrow
  C(\partial\Omega)$ is compact.
\end{lemma}
\begin{proof}
  First consider the operator $\RopHat\IErep$.
  Let $\dens\in C(\partial\Omega)$.  $\RopHat\IErep\dens$ evaluates
  the layer potential on the outer boundary $\partial\hat\Omega$.
  Since $x\mapsto \IErep\dens(x)$ is harmonic, $\IErep\dens(x)$ is
  analytic for $x\not\in \partial \Omega$~\cite[Thm.~6.6]{Kress:2014},
  and the restriction to the boundary $\partial\hat \Omega$,
  $\RopHat\IErep\dens$, is at least continuous.

  Next, consider the composite operator $\dens \mapsto
  \applyFEsolve{0}{\RopHat\IErep\dens}$. The boundary
  value problem
  \begin{align}
    \label{eq:coupling-classical}
    \quad -\triangle w(x) & = 0, & \quad x \in \hat\Omega, \\
    w(x) & = \RopHat\IErep\dens(x), &  x \in \partial\hat\Omega
    \nonumber
  \end{align}
  has a
  classical solution $w\in C^0(\overline{ \hat \Omega })\cap
  C^2(\hat\Omega)$ \cite[Thm.~6.13]{gilbarg_elliptic_2015}
  because of the regularity of the domain and data.
  More precisely, even $w\in C^\infty(\hat \Omega)$ by \cite[Thm.~6.17]{gilbarg_elliptic_2015}.

  The classical solution $w$ found above is identical to the unique
  (\cite[Cor.~8.2]{gilbarg_elliptic_2015})
  variational solution $\applyFEsolve{ 0 }{\RopHat\IErep \gamma}\in
  H^1(\hat\Omega)$.
  The classical maximum principle (e.g.~\cite[Thm.~3.1]{gilbarg_elliptic_2015}) then yields
  that
  \[
    \|\Rop\applyFEsolve{0}{\RopHat\IErep\dens}\|_\infty \le
    \|\RopHat\IErep\dens\|_\infty.
  \]
  Consequently, we have that
  $\Rop\FEsolve$ is bounded and $\RopHat\IErep$ is compact.
   The composition of a compact operator with a bounded operator is compact, which completes the proof.
\end{proof}

Using slightly different machinery, a discrete version of
Lemma~\ref{lem:compactcoupling} is available at least in $\mathbb R^2$.
To this end, let $\hat \Omega\subset \mathbb R^2$ be convex
and polygonal and define a finite
element subspace $V^h\subset H^1(\hat \Omega)$ of continuous
polynomials of degree $\ge 1$ on a quasi-uniform
triangulation of $\hat \Omega$ (in the sense of~\cite{Schatz:1980}).
Also define $V^h_0:=H^1_0(\hat \Omega)\cap V^h$.
Further define the \emph{discrete lifting operator} $\lifth:H^{1/2}(\partial
\hat \Omega)\to V^h$ and the
\emph{discrete solution operator}
$
  \FEsolveh: V^h \times H^{1/2}(\partial\hat\Omega) \to V^h
$
where $\applyFEsolveh{\zeta}{ \hat\rho}$ is defined
as the function $\tilde\mu^h+\lifth \hat\rho$, where $\tilde\mu^h\in V^h_0$ satisfies
\[
  \FEop(v^h) (\tilde\mu^h + \lifth \hat\rho) = \mathcal M (v^h) \zeta,
  \quad v^h \in V^h_0.
\]
(Again, the precise choice of the discrete lifting operator within these
guidelines has no influence on the obtained solution.)

\begin{theorem}  \label{thm:compactcoupling-discrete}
  Assume that $\hat\Omega\subset \mathbb R^2$ is bounded, convex, and
  polygonal and contains a domain $\Omega$ with a $C^\infty$ boundary. Further assume that
  $d(\partial\hat\Omega, \partial \Omega)> \epsilon$ for some finite
  $\epsilon>0$. Let the family of operators
  \[
    {\{\RopHat\IEreph:C(\partial\Omega) \rightarrow C(\partial\hat\Omega)\}}_h
  \]
  be collectively compact and the functions in their ranges be harmonic.
  Then the family of operators
  \[
    {\{\dens\mapsto \Rop\applyFEsolveh{0}{\RopHat\IEreph \dens }: C(\partial\Omega) \rightarrow
    C(\partial\Omega)\}}_h
  \]
  is collectively compact for sufficiently small $h$.
\end{theorem}
\begin{proof}
  First consider the operator $\RopHat\IEreph$.
  Let $\dens\in C(\partial\Omega)$.
  $\RopHat\IEreph\dens$ evaluates
  the layer potential on the outer boundary $\partial\hat\Omega$.
  Since $x\mapsto \IEreph\dens(x)$ is harmonic, it is also
  analytic for $x\not\in \partial \Omega$~\cite[Thm.~6.6]{Kress:2014},
  and so is its restriction $\RopHat\IEreph\dens$ to the
  boundary $\partial\hat \Omega$ is at least continuous.

  The discrete maximum principle~\cite{Schatz:1980}
  yields that
  \begin{equation}
    \|\Rop\applyFEsolveh{0}{\RopHat\IEreph\dens}\|_\infty \le
    C\|\RopHat\IEreph\dens\|_\infty.
    \label{eq:discrete-max-principle}
  \end{equation}
  where $C$ is independent of $h$.
  Noting $\Rop\applyFEsolveh{0}{\RopHat\IEreph\dens}\in V_h\subset C(\hat\Omega)$
  by construction, we have that
  $\Rop\applyFEsolveh{0}{\RopHat\IEreph} : C(\partial\Omega) \rightarrow
  C(\partial\Omega)$, with bounded $\Rop\FEsolveh$ and
  compact $\RopHat\IEreph$.
  We obtain our claim since the composition of a compact operator with a bounded operator is compact,
  noting that collective compactness follows from the
  $h$-independence of the constant in~\eqref{eq:discrete-max-principle}.
\end{proof}

The form of the operator in~\eqref{eq:densitywithFE} is
\begin{equation}\label{eq:Z}
\mathcal Z = CI + \IErepBdry - \Rop\applyFEsolve{0}{\RopHat\IErep}.
\end{equation}
Thus,
Lemma~\ref{lem:compactcoupling}
and Theorem~\ref{thm:compactcoupling-discrete} establish that the integral
equation~\eqref{eq:densitywithFE} is of the second kind and its discretization takes a form to which Anselone's Theorem applies. The operator $\mathcal Z$ in~\eqref{eq:Z} and its discrete version are the sum of an identity and a
compact operator.
Consequently, by the Fredholm alternative, if the
operator has no nullspace, then existence of the solution is
guaranteed. Again, convergence of the solution as $h\to 0$ is assured by Anselone's theorem.

We highlight some factors that influenced our choice of the IE representation.
The purely IE part of the operator,
$CI + \IErepBdry$, represents the behavior on
$\partial\Omega$ of a harmonic function exterior to $\partial \Omega$,
while the coupled FE part,
$\Rop\applyFEsolve{0}\RopHat\IErep$, is approximating a harmonic function
interior to $\partial \hat\Omega$; both functions have the same value on the
boundary $\partial\hat\Omega$ (but not on $\partial\Omega$).
A nontrivial nullspace exists in~\eqref{eq:densitywithFE} if
the intersection of the ranges of the operators is nontrivial.
Distinct decay behavior
of interior and exterior Dirichlet solutions generally keeps the
ranges of these operators from having a nontrivial intersection.

{}

\subsubsection{Remarks on the Behavior of the Error}
The observed convergence behavior is similar to that of the interior case,
but with additional components stemming from the FE error and the IE representation error in
the operator $\Rop\applyFEsolve{0}{\RopHat\IErep}$.   As it is a composition of operators
with known high-order accuracy, however, the composite scheme has the
same asymptotic error behavior as the less accurate of its
constituent parts, analogous to Lemma~\ref{lem:int-decoupling}.
In particular, the error in the $\applyFEsolve{0}{\RopHat\IErep}$ part of the overall
operator on $\dens$ is bounded by the error in its boundary
conditions---the operator error on $\RopHat\IErep$---by the weak
discrete maximum principle.  Thus the leading term effect of the
composition of the two is the same as the other FE or IE operator
error terms, depending on which error is larger.

The error behavior of the finite element solution $u_1$ once again follows standard finite element convergence theory,
with additional error incurred through error in $\RopHat\IErep\dens$ in the boundary condition.  However,
again the additional FE error is bounded by the error in $\RopHat\IErep\dens$ through the discrete weak maximum principle.  The net result
is that we expect to retain the same overall order of convergence as in the interior case and with similar dependence on the FE and IE solvers.
\subsection{Numerical experiments}
We consider the coupled system~\eqref{eq:int-ext-coupled} with the exact solution
\[
  u = \log(r_0) + 2\sin(\pi x)\sin(\pi y),
\]
where
\[
  r_0 = \sqrt{(x - 0.1)^2 + (y + 0.02)^2}.
\]
In each numerical example, GMRES is used to solve the linear system with algebraic multigrid preconditioning in the case of the FE operators.

The IE, FE, and coupled solutions are shown for a starfish exclusion in
Figure~\ref{fig:starfish-ext} and where the parametrization is given by
\begin{equation}
   \gamma(t) =
 \begin{bmatrix}
   1/2 + (1/8)\sin{(10\pi t)}\cos{(2\pi t)},\\
   1/2 + (1/8)\sin{(10\pi t)}\sin{(2\pi t)},
  \end{bmatrix}
  \quad t \in [0, 1].
  \label{eq:starfish}
\end{equation}
Figure~\ref{fig:int-ext-solution} gives a visual impression of the
obtained solution.
\ifpaper{Convergence results are found in Table~\ref{tab:int-ext-conv-starfish}.}

As expected, we observe high-order convergence.
\begin{figure}[!ht]
\begin{center}
\begin{subfigure}{0.33\textwidth}
\begin{center}
  \includegraphics[width=0.97\textwidth]{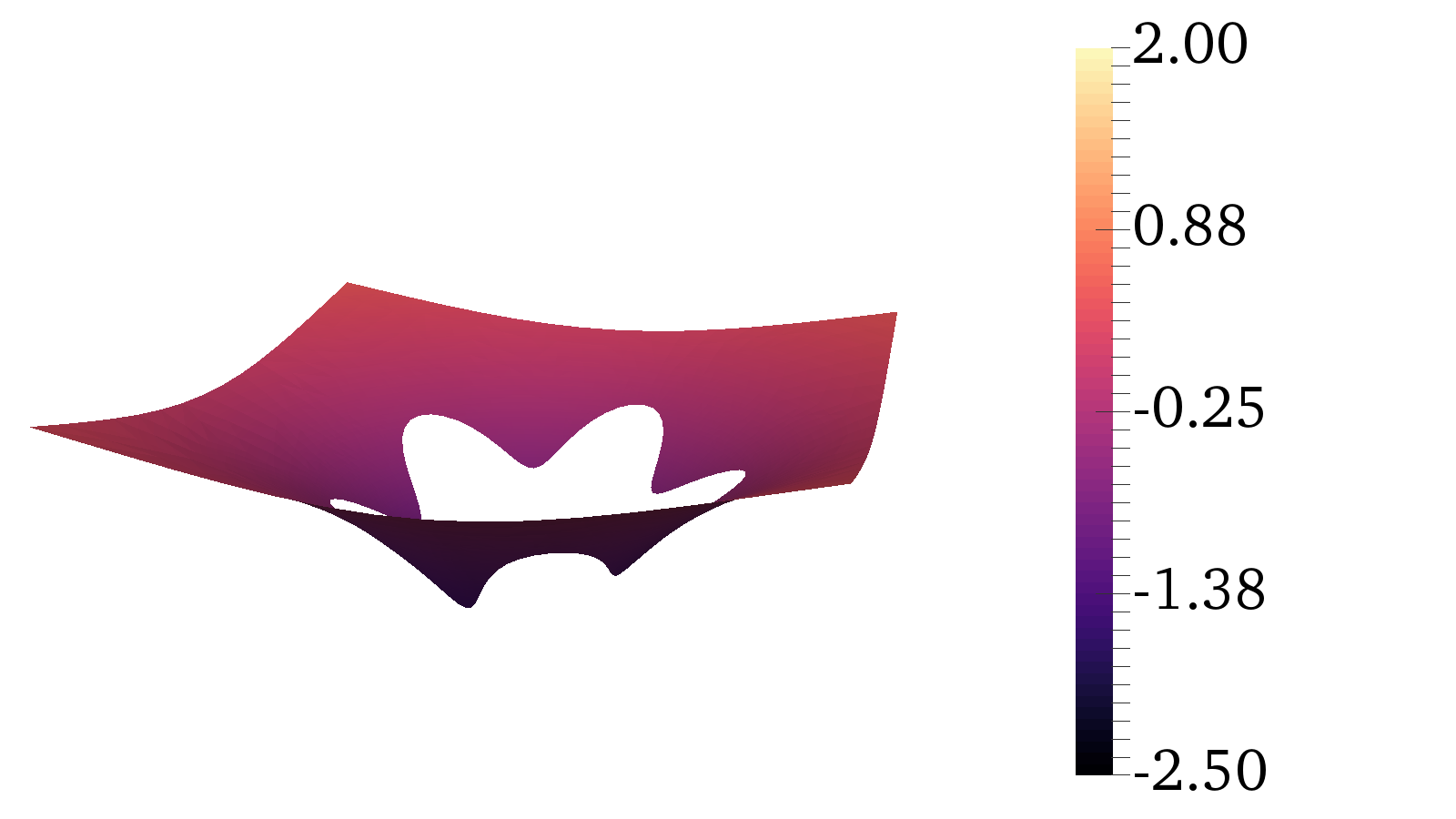}
  \caption{IE solution.}
\end{center}
\end{subfigure}\hfill\begin{subfigure}{0.33\textwidth}
\begin{center}
  \includegraphics[width=0.97\textwidth]{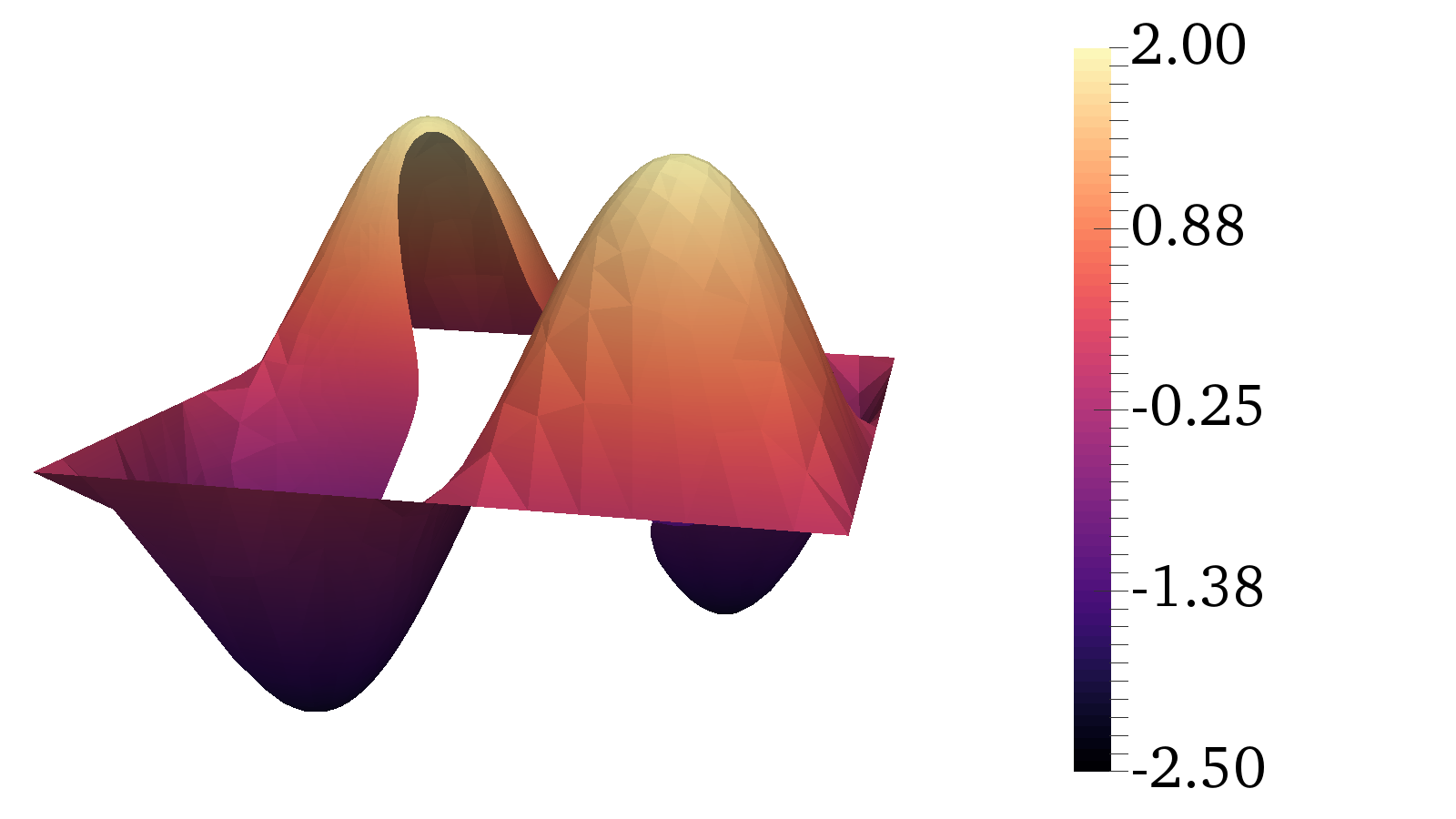}
  \caption{FE solution.}
\end{center}
\end{subfigure}
\begin{subfigure}{0.33\textwidth}
\begin{center}
  \includegraphics[width=0.99\textwidth]{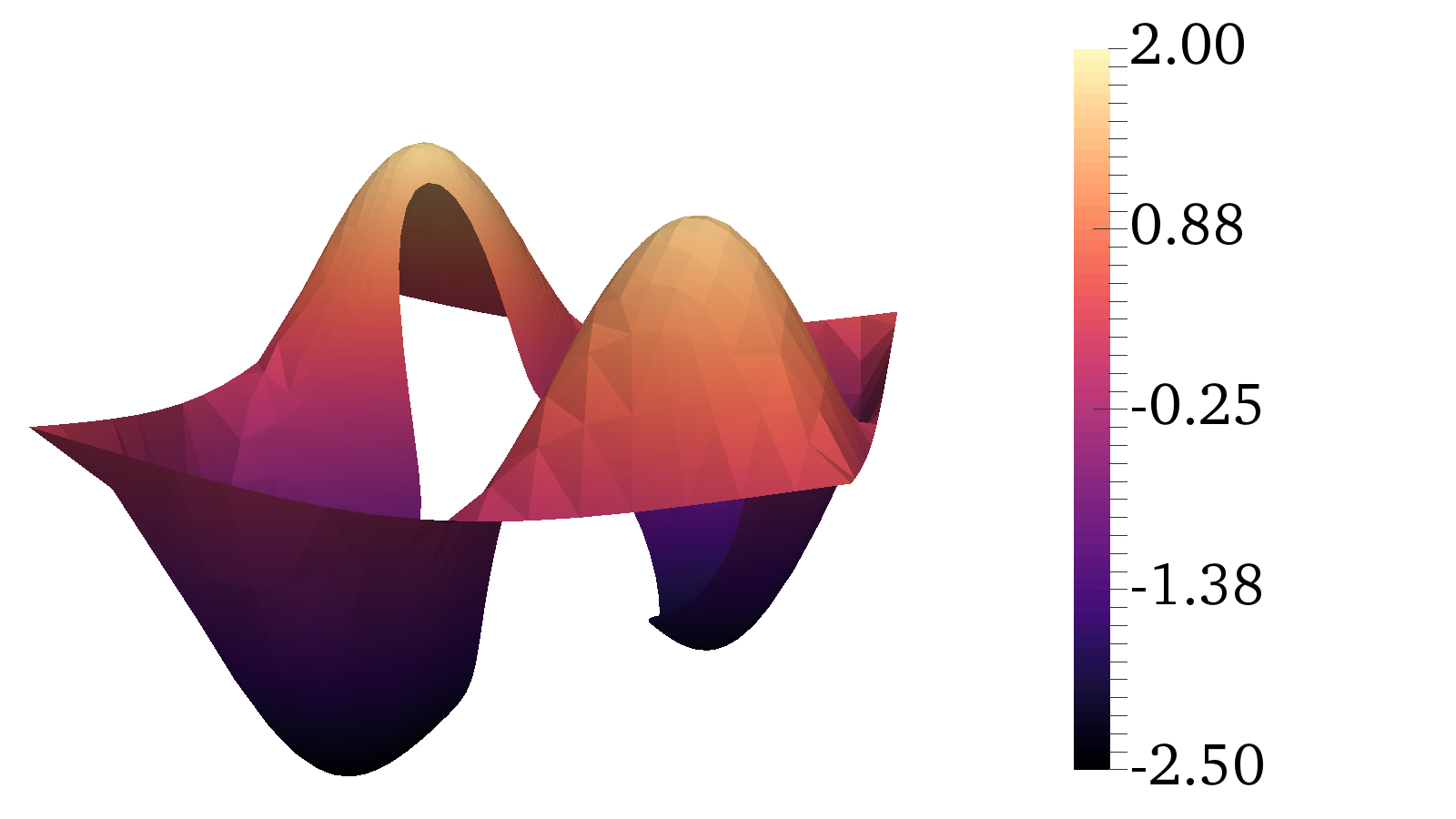}
  \caption{Total solution.}
\end{center}
\end{subfigure}
\caption{Individual and combined solutions on the true domain for the exterior embedded mesh problem.}~\label{fig:int-ext-solution}
\end{center}
\end{figure}
\ifpaper{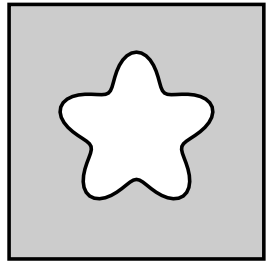
\begin{figure}[!ht]
  \begin{center}
  \resizebox{0.22\textwidth}{!}{      \begin{tikzpicture}
        \draw [fill=gray!40,thick](-1, -1) rectangle (1,1);
        \draw[thick, fill=white] plot file{starfish-curve.dat} -- cycle;
      \end{tikzpicture}}
      \caption{Starfish domain excluded from $\hat\Omega = [-1, -1] \times [-1, 1]$}\label{fig:starfish-ext}
  \end{center}
\end{figure}
}

\begin{table}
\begin{center}
\caption{Convergence of coupled interior-exterior FE-IE system for the excluded starfish domain.}
\label{tab:int-ext-conv-starfish}
\begin{tabular}{ c c  c c c c c }
\toprule
$p$ & $\pqbx$ & $\hfe$, $\hie$ & $\autoinfnorm{\hat\Omega\backslash\Omega}{\text{error}}$ & order &  $\| \text{error} \|_{L^2(\hat\Omega\backslash\Omega)}$ & order \\
\midrule
\multirow{4}{*}{1} & \multirow{4}{*}{2} & 0.133, 0.327 & 5.80\ee{-2} & --  & 2.38\ee{-2}  & -- \\
                     &                      & 0.067, 0.170 & 1.28\ee{-2} & 2.2 & 6.08\ee{-3} & 2.0 \\
                     &                      & 0.033, 0.085 & 3.47\ee{-3} & 1.9 & 1.47\ee{-3} & 2.0 \\
                     &                      & 0.017, 0.043 & 8.46\ee{-4} & 2.0 & 3.74\ee{-4} & 2.0 \\
\midrule
\multirow{4}{*}{2} & \multirow{4}{*}{3} & 0.133, 0.327 & 1.21\ee{-2} & --  & 2.05\ee{-3}  & -- \\
                     &                      & 0.067, 0.170 & 2.82\ee{-3} & 2.1 & 3.29\ee{-4} & 2.6 \\
                     &                      & 0.033, 0.085 & 5.53\ee{-4} & 2.3 & 3.70\ee{-5} & 3.2 \\
                     &                      & 0.017, 0.043 & 6.48\ee{-5} & 3.1 & 3.65\ee{-6} & 3.3 \\
\midrule
\multirow{4}{*}{3} & \multirow{4}{*}{4} & 0.133, 0.327 & 1.00\ee{-2} & --  & 1.12\ee{-3}  & -- \\
                     &                      & 0.067, 0.170 & 1.73\ee{-3} & 2.5 & 9.97\ee{-5} & 3.5 \\
                     &                      & 0.033, 0.085 & 1.91\ee{-4} & 3.2 & 7.19\ee{-6} & 3.8 \\
                     &                      & 0.017, 0.043 & 1.47\ee{-5} & 3.7 & 4.08\ee{-7} & 4.1 \\
\bottomrule
\end{tabular}
\end{center}
\end{table}

\section{FE-IE for Interface Problems}\label{sec:interface}
In this section,
we combine the elements described in
Sections~\ref{sec:int-poisson} and~\ref{sec:int-ext} to form a new
embedded boundary method for the interface problem
(\ref{eq:interface}).  In our approach, a fictitious
domain $\hat\Omega^i$ is introduced so that $\Omega^i \subset \hat\Omega^i$,
defining $\hat\Omega^e$ as $\Omega^e \cup \Omega^i$. Then the
problem is separated into two subproblems with an appropriate FE-IE splitting on
each.  This is illustrated in Figure~\ref{fig:interface-schematic}.
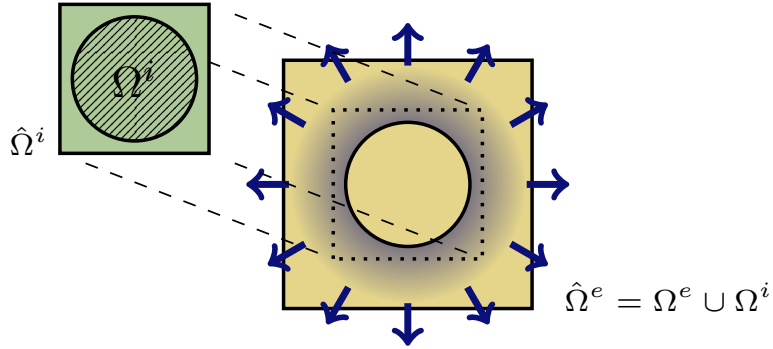
\begin{figure}[ht!]
\begin{center}
 \resizebox{0.65\textwidth}{!}{
  \begin{tikzpicture}
 \def\centrx{-2.2};
  \def\centry{0.85};
    \def\radlen{1.3};
           \draw [fill=\extHatColor!60, thick](-1,-1) rectangle (1,1);
        \draw[->, ultra thick, \extArrowColor] (0,0) -- (\radlen, 0);
    \draw[->,  ultra thick, \extArrowColor] (0,0) -- (\radlen*0.866, 0.5*\radlen);
    \draw[->,  ultra thick, \extArrowColor] (0,0) -- (0.5*\radlen, \radlen*0.866);
    \draw[->,  ultra thick,  \extArrowColor] (0,0) -- (0, \radlen);
    \draw[->,   ultra thick, \extArrowColor] (0,0) -- (-\radlen*0.866, 0.5*\radlen);
    \draw[->,  ultra thick,  \extArrowColor] (0,0) -- (-0.5*\radlen, \radlen*0.866);
    \draw[->,  ultra thick,  \extArrowColor] (0,0) -- (-\radlen, 0);
    \draw[->,  ultra thick,  \extArrowColor] (0,0) -- (-\radlen*0.866, -0.5*\radlen);
    \draw[->,  ultra thick,  \extArrowColor] (0,0) -- (-0.5*\radlen, -\radlen*0.866);
    \draw[->, ultra thick,   \extArrowColor] (0,0) -- (0, -\radlen);
    \draw[->, ultra thick,  \extArrowColor] (0,0) -- (\radlen*0.866, -0.5*\radlen);
    \draw[->,  ultra thick,  \extArrowColor] (0,0) -- (0.5*\radlen, -\radlen*0.866);
   \filldraw[\extHatColor!60, even odd rule,inner color=\extArrowColor,outer color=\extHatColor!60] (0, 0) circle [radius=0.95];
        \draw [fill=\extHatColor!60, thick] (0., 0.) circle [radius=0.5];
          \draw [-, dotted, thick](- 0.6, - 0.6) rectangle (0.6, 0.6);
    \node at (2.1, -0.9) {\scriptsize$\hat\Omega^e = \Omega^e \cup \Omega^i$};
      \node (1) at (-0.6, 0.6) {};
   \node (2) at (0.6, 0.6) {};
   \node (3) at (0.6, -0.6) {};
   \node (4) at (-0.6, -0.6) {};
     \node (5) at (\centrx -0.6, \centry + 0.6) {};
   \node (6) at (\centrx + 0.6, \centry + 0.6) {};
   \node (7) at (\centrx + 0.6, \centry - 0.6) {};
   \node (8) at (\centrx - 0.6, \centry - 0.6) {};
       \draw (1) [ -, dashed] to  (5);
      \draw (2) [ -, dashed] to  (6);
      \draw (3) [ -, dashed] to  (7);
      \draw (4) [ -, dashed] to  (8);

      \draw [fill=\intHatColor!40,thick](\centrx - 0.6,\centry - 0.6) rectangle (\centrx + 0.6,\centry + 0.6);
    \draw[pattern=\intPatternType, thick, pattern color=\intPatternColor] (\centrx, \centry) circle [radius=0.5];
     \node at (\centrx, \centry) {$\Omega^i$};
     \node at (\centrx - 0.85, \centry -0.5) {\scriptsize$\hat\Omega^i$};

      \end{tikzpicture}}
   \caption{Coupled subproblems for an embedded interface.}\label{fig:interface-schematic}
\end{center}
\end{figure}

There are four components to the combined solution: $u_1^i :
\hat\Omega^i\to \mathbb R$ and $u_2^i : \Omega^i \to \mathbb R$
for the interior solution, plus $u_1^e : \hat\Omega^e \to \mathbb R $
and $u_2^e: \mathbb{R}^d\backslash\Omega^i\to \mathbb R$ for the exterior solution.

As before, $u_1^i$ and $u_1^e$ represent the finite element components
of the solution.
For the interior and exterior integral equation solutions,
we choose the combined representation
\[
  u_2^i = \alpha_1 \DblOp\dens^i + \alpha_2 \SnglOp\dens^e,
  \quad
  \text{and}
  \quad
  u_2^e = \alpha_3 \DblOp\dens^i + \alpha_4 \SnglOp\dens^e,
\]
for some constant coefficients $\alpha_j$.  We next determine
$\alpha_j$ to ensure that all integral operators are of the second
kind.

Taking the limits of these expressions at the interface and adding the interface restrictions of the
finite element solutions gives the following form for the interface conditions:
\begin{equation}
  \Rop u_1^i  +   \alpha_1 \left(\DblOp_{-}\right) \dens^i
              +   \alpha_2 \left(\SnglOp_{-}\right) \dens^e =
c \dRop u_1^e + c \alpha_3 \left(\DblOp_{+}\right) \dens^i
              + c \alpha_4 \left(\SnglOp_{+}\right)\dens^e + a(x),
\end{equation}
and
\begin{equation}
\dRop u_1^i       + \alpha_1\partial_n \left(\DblOp_{-}\right) \dens^i
                  + \alpha_2\partial_n \left(\SnglOp_{-}\right)\dens^e =
\kappa\dRop u_1^e + \kappa\alpha_3\partial_n  \left(\DblOp_{+}\right)\dens^i
                  + \kappa\alpha_4 \partial_n \left(\SnglOp_{+}\right)\dens^e + b(x),
\end{equation}
where $\DblOp_{\pm}$ indicates the interior ($-$) or exterior ($+$) limit of the double-layer operator; similarly for $\SnglOp_{\pm}$.
The $\dRop$ operator is defined analogous to $\Rop$, but for gradient of the FE solution normal to the interface $\Gamma$.

\ifpaper{  We use the well-known jump conditions for layer potentials~\cite{Kress:2014} and collect the terms on each operator, which yields
 }
\begin{equation}\label{eq:collected-value-cond}
  \Rop u_1^i + \left[-\frac{\alpha_1 + c\alpha_3}{2}I + (\alpha_1 - c\alpha_3)\DblOpbd\right]\dens^i  =
      c\Rop u_1^e +  \left[-\alpha_2 + c\alpha_4\right]\SnglOpbd\dens^e + a(x)
\end{equation}
and
\begin{equation}\label{eq:collected-deriv-cond}
  \dRop u_1^i + \left[\frac{\alpha_2 + \kappa\alpha_4}{2}I + (\alpha_2 - \kappa\alpha_4)\SpOpbd\right]\dens^e
  = \kappa\dRop u_1^e + \left[ - \alpha_1 + \kappa\alpha_3\right] \DpOpbd\dens^i  + b(x).
\end{equation}

To determine suitable values for $\alpha_j$,  first we eliminate
the hypersingular operator $\DpOpbd$
from~\eqref{eq:collected-deriv-cond}, necessitating
\begin{equation}\label{eq:req-one}
\alpha_1 = \kappa\alpha_3.
\end{equation}
With $\DpOpbd$ eliminated,~\eqref{eq:collected-deriv-cond} is an equation for $\dens^e$ with operator
\[  \frac{\alpha_2 + \kappa\alpha_4}{2}I + (\alpha_2 - \kappa\alpha_4)\SpOpbd. \]
In order to have an operator with only the trivial nullspace, guaranteeing a unique solution for $\dens^e$, we
select the coefficients of $I$ and $\SpOpbd$ to have opposite sign.

Next consider the jump condition~\eqref{eq:collected-value-cond}.
Enforcing the requirement~\eqref{eq:req-one}, the operator on
$\dens^i$ is
\begin{equation}\label{eq:sigmaiop}
  -\frac{(\kappa + c)\alpha_3}{2}I + (\kappa - c)\alpha_3 \DblOpbd.
\end{equation}
This results in three possibilities based on $\kappa$ and $c$:
\begin{enumerate}
  \item $\kappa \neq c$ and $\kappa \neq -c$, where both terms remain,
  \item $\kappa = c$, where the double-layer term drops out, and
  \item $\kappa = -c$, where the identity term is eliminated.
\end{enumerate}

We consider each case in the following sections.
\subsection{The case \texorpdfstring{$\kappa \neq c$}{k!=c} and \texorpdfstring{$\kappa \neq -c$}{k!=c}}
Combining the interface condition equations \eqref{eq:collected-value-cond} and \eqref{eq:collected-deriv-cond} with
the condition \eqref{eq:req-one} and the interior and exterior FE problems for $u_1^i$ and $u_1^e$ yields
a coupled system for the interface problem:
with $\FEop^i(v)$ and $v \in H^1_0(\hat\Omega^i)$ for the interior problem and $\FEop^e(w)$
and $w \in H^1_0(\hat\Omega^e)$ for the
 coupled exterior problem, we seek $(\dens^i, \dens^e) \in C(\Gamma)$, $u_1^i \in H_0^1(\hat\Omega^i)$,
 $\tilde{u}_1^e \in H_0^1(\hat\Omega^e)$, and $\hat r_1 \in H^{1/2}(\partial\hat\Omega^e)$ such that
\begin{equation}\label{eq:generalmat-split}
 \mathcal{C}
 \begin{roomymat}{c}
   \dens^i \\
   \dens^e \\
   \hline
   u_1^i \\
  \tilde{u}_1^e\\
  \hat r_1^e  \end{roomymat}
    =
  \begin{roomymat}{c}
    a(x) \\
    b(x) \\
    \hline
    \mathcal{M}^i(v) \, f^i\\
   \mathcal{M}^e(w)\, f^e \\
    \hat g
\end{roomymat}
\qquad
\textnormal{for all $v \in H^1_0(\hat{\Omega}^i), w \in H_0^1(\hat{\Omega}^e)$,}
\end{equation}
where
\begin{equation}
  \mathcal{C} =
\begin{roomymat}{c c | c c c}
-\frac{1}{2}(\kappa+c)\alpha_3I + (\kappa - c)\alpha_3\DblOpbd &  (\alpha_2 - c\alpha_4)\SnglOpbd  & \Rop & -c\Rop & -c\Rop\lift^e  \\
 0         &   \frac{1}{2}(\alpha_2 + \kappa\alpha_4)I + (\alpha_2 - \kappa\alpha_4)\SpOpbd & \dRop & -\kappa\dRop & \kappa\dRop\lift^e \\
 \hline
 0 & 0 &  \FEop^i(v) & 0 & 0 \\
 0 & 0 &    0 & \FEop^e(w) & \FEop^e(w)\lift^e \\
 \alpha_3\RopHat^e\DblOp  &  \alpha_4\RopHat^e\SnglOp &   0 & 0  & I \\
 \end{roomymat}.
 \end{equation}
 The lifting operator $\lift^e$ is as described in Section \ref{sec:FE-IE-decomp} and acts on the boundary $\partial\hat\Omega^e$.  The source functions $f^i:\hat\Omega^i\to \mathbb R$ and $f^e:\hat\Omega^e \to \mathbb R$
are once again suitably restricted and/or extended versions of the
right-hand side $f$.
From this we see that
$a(x)$ and
$b(x)$ in the jump conditions are handled as additional
terms on the right-hand side of the system.

\subsubsection{A Solution Procedure Involving the Schur Complement}
For the exterior problem, we apply a Schur complement to~\eqref{eq:generalmat-split}.
To simplify notation, we apply~\eqref{eq:req-one} and define the block IE operator as
\begin{equation*}
  \tensor{\mathcal{I}} + \tensor{\IErepBdry} = \begin{bmatrix}-\frac{1}{2}(\kappa+c)\alpha_3I + (\kappa - c)\alpha_3\DblOpbd &  (\alpha_2 - c\alpha_4)\SnglOpbd  \\
    0            &   \frac{1}{2}(\alpha_2 + \kappa\alpha_4)I + (\alpha_2 - \kappa\alpha_4)\SpOpbd  \end{bmatrix}.
\end{equation*}
We also define the block coupling operator
\begin{equation*}
  \tensor{\Rop} =      \begin{bmatrix}    \Rop       &    -c\Rop  \\
                                              \dRop   & -\kappa\dRop  \end{bmatrix}.
\end{equation*}
As restriction to the outer boundary $\partial\hat\Omega^e$ is only necessary for $u_2^e$,
we do not need a block form of $\RopHat$.  Rather,
\[ \RopHat^e\IErep^e =\begin{bmatrix}\alpha_3\RopHat^e\DblOp  &  \alpha_4\RopHat^e\SnglOp \end{bmatrix} \]
 will suffice.

Finally, we define the
block FE solution operator $\tensor\FEsolve : L^2(\hat\Omega^i) \times L^2(\hat\Omega^e) \times   H^{1/2}(\partial\hat\Omega^e) \rightarrow H_0^1(\hat\Omega^i) \times H^1(\hat\Omega^e)$
 such that the $\mu^i \in H_0^1(\hat\Omega^i)$ and $\tilde{\mu}^e \in H^1_0(\hat\Omega^e)$
 defined by  $\hvectwo{\mu^i}{\tilde{\mu}^e + \lift^e\hat \rho} = \applyblockFEsolve{\zeta^i}{\zeta^e}{\hat\rho}$
 satisfy
\begin{align}
    \FEop^i(v) \mu^i &  = \mathcal{M}^i(v) \zeta^i &\forall v \in H_0^1(\hat\Omega^i), \label{eq:blockFEsolve} \\  \nonumber
    \FEop^e(w)(\tilde{\mu}^e + \lift^e\hat\rho) & = \mathcal{M}^e(w) \zeta^e  & \forall w \in H_0^1(\hat\Omega^e).
 \end{align}
We then write the equations for the density functions as
  \begin{equation}
    \left\{ \tensor{\mathcal{I}} + \tensor{\IErepBdry} -
    \tensor{\Rop}\,\tensor\FEsolve\begin{bmatrix}0 \\ 0 \\ \RopHat^e\IErep^e \end{bmatrix} \right \}
   \begin{roomymat}{c}
     \dens^i \\
     \dens^e \end{roomymat} = \begin{roomymat}{c}
    a(x) \\ b(x) \end{roomymat} -
    \tensor{\Rop}\,\tensor{\FEsolve}\begin{bmatrix}  f^i  \\ f^e \\ \hat{g}\end{bmatrix}.
     \label{eq:interface-split-solve-IE}
  \end{equation}
Once the densities  are known, the FE solutions are defined as $\applyblockFEsolve{f^i}{f^e}{\hat g - \RopHat^e\IErep^e}\vecdens$.

The equation~\eqref{eq:interface-split-solve-IE}
very closely resembles~\eqref{eq:densitywithFE} for the exterior case in Section
\ref{sec:int-ext}.  The main difference here---apart from the doubling
of the number of variables---is the appearance of $\dRop$
terms in $\tensor{\Rop}$ and thus in the operator on $\vecdens$ in (\ref{eq:interface-split-solve-IE}).  As the output from $\RopHat^e\IErep^e\vecdens$ is
smooth, however, the result of $\applyblockFEsolve{0 }{0}{\RopHat^e\IErep^e}\vecdens$ is smooth for most domains, as it is approximating a harmonic function for $u_1^e$ and will return $u_1^i = 0$ due to
the homogeneous boundary conditions enforced on $u_1^i$.  We expect this to mitigate any negative effects on the conditioning of the numerical system.

\subsection{The case \texorpdfstring{$\kappa = c$}{k=c}}

In this case, the $\DblOpbd$ term drops out of the operator on $\dens^i$.   We may choose $\alpha_2 = c\alpha_4 = \kappa\alpha_4$
which results in IE block in the form of a scaled identity:
\begin{equation}
 [\kappa = c \text{ case}] \qquad  \tensor{\mathcal{I}} + \tensor{\IErepBdry} = \begin{bmatrix}   -\frac{1}{2}(\kappa+c)\alpha_3 I  & 0  \\
                       0            &   \kappa\alpha_2 I  \end{bmatrix}.
 \end{equation}

\subsection{The case \texorpdfstring{$\kappa = -c$}{k=c}}

In this case, the operator on $\dens^i$ as defined
in~\eqref{eq:sigmaiop} loses the identity term, resulting in an
equation for $\dens^i$ which is not of the second kind.  We leave this
to future work.
\subsection{Numerical experiments}
We demonstrate the behavior of the method on two test
problems, one with $\kappa \neq c$ and one with $\kappa = c$.
The choices for the coefficients of the layer potential
representation are summarized in Table~\ref{tab:alphas}.

\begin{table}[!ht]
\begin{center}
  \caption{Choices of coefficients for numerical experiments}\label{tab:alphas}
  \begin{tabular} {r c c c c c }
   \toprule
Case             & $\alpha_1$       & $\alpha_2$ & $\alpha_3$       & $\alpha_4$ \\
                 \midrule
$\kappa \neq c$: & $\kappa\alpha_3$ & 0          & $1/(\kappa - c)$ & $1/\kappa$ \\

$\kappa = c$:    & $\kappa\alpha_3$ & 1          & $-1/\kappa$      & $1/\kappa$ \\

    \bottomrule
 \end{tabular}
\end{center}
\end{table}

\subsubsection{The quadratic-log test case}

In this manufactured-solution experiment,
the interface $\Gamma$ is a circle of radius $0.5$ and $\Omega^e \cup
\Omega^i$ is the square $[-1, -1] \times [-1, 1]$.  Relevant data for
this test case is given in Table~\ref{tab:testdata} with
the solution shown in Figure~\ref{fig:quadsine}. We note that the
right-hand side $f$ in this example exhibits a discontinuity, however
the reduction in convergence order discussed in
Section~\ref{sec:intro} does not occur because, by way of their given
expressions, both may be smoothly extended into the respective other
domain.
\begin{figure}[!ht]
    \centering
    \begin{subfigure}[b]{0.45\textwidth}
        \includegraphics[width=\textwidth]{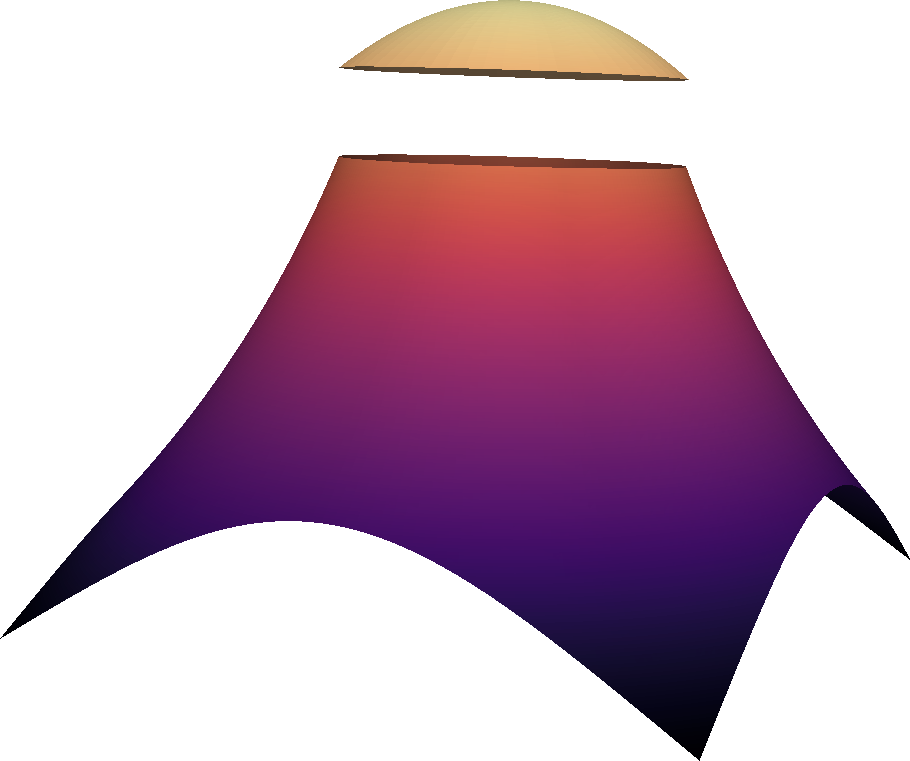}
    \caption{quadratic-log}\label{fig:quad-log}
    \end{subfigure}
    \hfill
    \begin{subfigure}[b]{0.45\textwidth}
        \includegraphics[width=\textwidth]{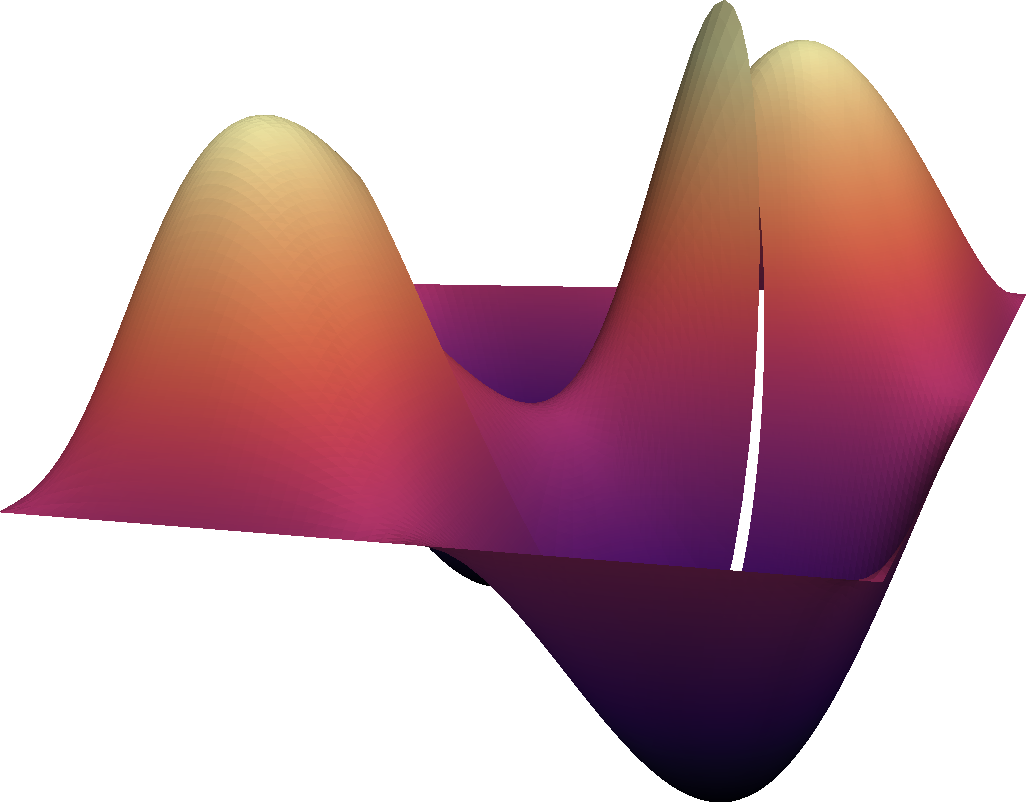}
  \caption{sine-linear}\label{fig:sine-sine}
    \end{subfigure}
    \caption{Numerical solutions for the test cases.}\label{fig:quadsine}
\end{figure}
\begin{table}
  \begin{tabular}{c | c c c | c c c | c c c c}
   \toprule
    \multirow{2}{*}{case} &
   \multicolumn{3}{c|}{interior} &
    \multicolumn{3}{c|}{exterior} &
    \multicolumn{4}{c}{interface} \\
      &
     $u$ & $f$ & $\hat{g}$ &
     $u$ & $f$ & $\hat{g}$ &
     $\kappa$ &$c$ & $a(x)$ & $b(x)$\\
     \midrule
quadratic-log &
   $-\frac{5}{6} r^2$ & $\frac{10}{3}$ & --- &
   $-\frac{5}{4}\log{\left(\frac{1}{2r}\right)} - \frac{11}{24}$ & 0 & $u^*$ &
   1/3 & 1 & 1/4 & 0   \\
sine-linear &
  $s_x s_y + x + y$ & $4\pi^2s_x s_y$ & ---   &
  $s_x s_y$         & $4\pi^2s_x s_y$ & $u^*$ &
  1 & 1 & $x + y$ & ${\bf 1} \cdot \nvec$ \\
   \bottomrule
  \end{tabular}
  \caption{Data for the quadratic-log and sine-linear test cases.  $s_x$ and $s_y$ denote
  $\sin{(2\pi x)}$ and $\sin{(2\pi y)}$.}\label{tab:testdata}
\end{table}

Convergence
is shown in
\ifpaper{Table~\ref{tab:quad-log-error}.  We}

achieve high-order convergence for the FE basis functions $p \geq 2$ and
attribute the lower convergence order to the fact that the
solution now depends on the FE derivative, and as a result
we expect to lose an order of convergence in the gradient representation compared to the
solution representation.

An interesting artifact arises in the convergence for $p = 1$: there is negative convergence between the coarsest and second-coarsest meshes.
This is due to the appearance of the derivative of the finite element
solution in the system.  For $p = 1$, the derivative of the FE
solution is piecewise constant.  On a coarse mesh, piecewise constants
poorly represent even smooth solutions with variation. As a result,
the $p = 1$ case is sensitive to element placement, especially
for coarse meshes.  Because of this effect we disregard data for $p=1$ and recommend using at least $p=2$.

\begin{table}[!ht]
\begin{center}
\caption{Convergence data for quadratic-log test problem with $\hat\Omega^i = [-0.6, -0.6] \times [0.6, 0.6]$}
\label{tab:quad-log-error}
\begin{tabular}{ c c c c c c c c }
\toprule
$p$ & $\pqbx$ & solution & $\hfe$, $\hie$ & $\infnormvol{\text{error}}$ & order &  $\| \text{error} \|_{L^2(\Omega)}$ & order \\
\midrule
\multirow{4}{*}{1} & \multirow{4}{*}{2} & \multirow{4}{*}{interior}  & 0.080, 0.224 & 6.93\ee{-3} & --  & 3.81\ee{-3}  & -- \\
                     &                      &                        & 0.040, 0.108 & 1.43\ee{-2} & -1.0 & 1.14\ee{-2} & -1.5 \\
                     &                      &                        & 0.020, 0.053 & 5.48\ee{-3} & 1.4 & 3.88\ee{-3} & 1.5 \\
                     &                      &                        & 0.010, 0.026 & 1.51\ee{-3} & 1.8 & 8.11\ee{-4} & 2.2 \\
\cline{3-8}
\multirow{4}{*}{1} & \multirow{4}{*}{2} & \multirow{4}{*}{exterior} & 0.080, 0.224 & 8.53\ee{-3} & --  & 5.61\ee{-3}  & -- \\
                     &                      &                       & 0.040, 0.108 & 1.51\ee{-2} & -0.8 & 1.05\ee{-2} & -0.9 \\
                     &                      &                       & 0.020, 0.053 & 5.34\ee{-3} & 1.5 & 3.43\ee{-3} & 1.6 \\
                     &                      &                       & 0.010, 0.026 & 1.48\ee{-3} & 1.8 & 7.23\ee{-4} & 2.2 \\
\midrule
\multirow{4}{*}{2} & \multirow{4}{*}{3} & \multirow{4}{*}{interior}  & 0.080, 0.224 & 3.64\ee{-3} & --  & 2.57\ee{-3}  & -- \\
                     &                      &                        & 0.040, 0.108 & 6.41\ee{-4} & 2.4 & 3.88\ee{-4} & 2.6 \\
                     &                      &                        & 0.020, 0.053 & 7.57\ee{-5} & 3.0 & 4.34\ee{-5} & 3.1 \\
                     &                      &                        & 0.010, 0.026 & 1.05\ee{-5} & 2.8 & 5.64\ee{-6} & 2.9 \\
\cline{3-8}
\multirow{4}{*}{2} & \multirow{4}{*}{3} & \multirow{4}{*}{exterior} & 0.080, 0.224 & 5.14\ee{-3} & --  & 2.63\ee{-3}  & -- \\
                     &                      &                       & 0.040, 0.108 & 7.79\ee{-4} & 2.6 & 3.78\ee{-4} & 2.7 \\
                     &                      &                       & 0.020, 0.053 & 8.63\ee{-5} & 3.1 & 4.15\ee{-5} & 3.1 \\
                     &                      &                       & 0.010, 0.026 & 1.12\ee{-5} & 2.9 & 5.27\ee{-6} & 2.9 \\
\midrule
\multirow{4}{*}{3} & \multirow{4}{*}{4} & \multirow{4}{*}{interior}  & 0.080, 0.224 & 7.95\ee{-4} & --  & 4.22\ee{-4}  & -- \\
                     &                      &                        & 0.040, 0.108 & 8.04\ee{-5} & 3.1 & 3.83\ee{-5} & 3.3 \\
                     &                      &                        & 0.020, 0.053 & 6.34\ee{-6} & 3.6 & 2.75\ee{-6} & 3.7 \\
                     &                      &                        & 0.010, 0.026 & 4.52\ee{-7} & 3.8 & 1.91\ee{-7} & 3.8 \\
\cline{3-8}
\multirow{4}{*}{3} & \multirow{4}{*}{4} & \multirow{4}{*}{exterior} & 0.080, 0.224 & 1.16\ee{-3} & --  & 4.77\ee{-4}  & -- \\
                     &                      &                       & 0.040, 0.108 & 9.99\ee{-5} & 3.4 & 3.99\ee{-5} & 3.4 \\
                     &                      &                       & 0.020, 0.053 & 7.13\ee{-6} & 3.7 & 2.77\ee{-6} & 3.8 \\
                     &                      &                       & 0.010, 0.026 & 4.80\ee{-7} & 3.8 & 1.87\ee{-7} & 3.8 \\
\bottomrule
\end{tabular}
\vspace*{0.2in}
\begin{tabular}{c c}
    \resizebox{1.8in}{!}{\begin{tikzpicture}
   \path [fill=white](-1.2,-1) rectangle (2.3,1.4);
    \draw [fill=\intHatColor!40,thick]( - 0.6, - 0.6) rectangle (0.6,  0.6);
    \draw[pattern=\intPatternType, thick, pattern color=\intPatternColor] (0, 0) circle [radius=0.5];
     \node at (0, -0.8) {\scriptsize interior};
     \end{tikzpicture}}    &           \resizebox{1.8in}{!} {\begin{tikzpicture}
     \path[fill=white](-1.2,-1) rectangle(2.3, 1.4);
       \draw [fill=\extHatColor!60, thick](-1,-1) rectangle (1,1);
       \path [pattern=\intPatternType, pattern color=\intPatternColor](-1,-1) rectangle (1,1);
     \draw [fill=\extHatColor!60, thick] (0., 0.) circle [radius=0.5];
    \node at (1.7, -0.8) {\scriptsize exterior};
    \end{tikzpicture}}  \\
  \bottomrule
\end{tabular}
\end{center}
\end{table}

\subsubsection{The sine-linear test case}

Next, we consider a test case for which $\kappa = c$; its data is summarized
in Table~\ref{tab:testdata} and shown for a circular interface
in Figure~\ref{fig:sine-sine}.  This is once again a manufactured solution with
the same forcing function $f$ through
$\Omega^i \cup \Omega^e$.  The extra linear function added to the
interior solution influences the numerical system only through the
non-homogeneous jump conditions $a(x)$ and $b(x)$.
Convergence data is shown in
\ifpaper{Table~\ref{tab:sine-sine-starfish-error} for the starfish interface.}

Again, we see high-order convergence.

\begin{table}[!ht]
\begin{center}
\caption{Convergence data for sine-linear test problem with $\Omega^i = $ starfish curve \eqref{eq:starfish}}
\label{tab:sine-sine-starfish-error}
\begin{tabular}{ c c c c c c c c }
\toprule
$p$ & $\pqbx$ & solution & $\hfe$, $\hie$ & $\infnormvol{\text{error}}$ & order &  $\| \text{error} \|_{L^2(\Omega)}$ & order \\
\midrule
\multirow{4}{*}{1} & \multirow{4}{*}{2} & \multirow{4}{*}{interior}  & 0.080, 0.327 & 7.43\ee{-2} & --  & 1.66\ee{-2}  & -- \\
                     &                      &                        & 0.040, 0.170 & 5.54\ee{-3} & 3.7 & 1.27\ee{-3} & 3.7 \\
                     &                      &                        & 0.020, 0.085 & 1.37\ee{-3} & 2.0 & 3.17\ee{-4} & 2.0 \\
                     &                      &                        & 0.010, 0.043 & 3.63\ee{-4} & 1.9 & 7.87\ee{-5} & 2.0 \\
\cline{3-8}
\multirow{4}{*}{1} & \multirow{4}{*}{2} & \multirow{4}{*}{exterior} & 0.080, 0.327 & 7.32\ee{-2} & --  & 2.63\ee{-2}  & -- \\
                     &                      &                       & 0.040, 0.170 & 5.63\ee{-3} & 3.7 & 3.08\ee{-3} & 3.1 \\
                     &                      &                       & 0.020, 0.085 & 1.45\ee{-3} & 2.0 & 7.28\ee{-4} & 2.1 \\
                     &                      &                       & 0.010, 0.043 & 3.48\ee{-4} & 2.0 & 1.93\ee{-4} & 1.9 \\
\midrule
\multirow{4}{*}{2} & \multirow{4}{*}{3} & \multirow{4}{*}{interior}  & 0.080, 0.327 & 2.11\ee{-3} & --  & 3.91\ee{-4}  & -- \\
                     &                      &                        & 0.040, 0.170 & 3.41\ee{-4} & 2.6 & 2.64\ee{-5} & 3.9 \\
                     &                      &                        & 0.020, 0.085 & 1.76\ee{-5} & 4.3 & 2.69\ee{-6} & 3.3 \\
                     &                      &                        & 0.010, 0.043 & 2.90\ee{-6} & 2.6 & 6.97\ee{-7} & 1.9 \\
\cline{3-8}
\multirow{4}{*}{2} & \multirow{4}{*}{3} & \multirow{4}{*}{exterior} & 0.080, 0.327 & 3.01\ee{-3} & --  & 4.55\ee{-4}  & -- \\
                     &                      &                       & 0.040, 0.170 & 1.29\ee{-4} & 4.5 & 3.26\ee{-5} & 3.8 \\
                     &                      &                       & 0.020, 0.085 & 1.12\ee{-5} & 3.5 & 3.98\ee{-6} & 3.0 \\
                     &                      &                       & 0.010, 0.043 & 2.87\ee{-6} & 2.0 & 8.99\ee{-7} & 2.1 \\
\midrule
\multirow{4}{*}{3} & \multirow{4}{*}{4} & \multirow{4}{*}{interior}  & 0.080, 0.327 & 2.92\ee{-3} & --  & 2.74\ee{-4}  & -- \\
                     &                      &                        & 0.040, 0.170 & 4.78\ee{-4} & 2.6 & 1.21\ee{-5} & 4.5 \\
                     &                      &                        & 0.020, 0.085 & 1.78\ee{-5} & 4.7 & 3.33\ee{-7} & 5.2 \\
                     &                      &                        & 0.010, 0.043 & 1.25\ee{-6} & 3.8 & 1.04\ee{-8} & 5.0 \\
\cline{3-8}
\multirow{4}{*}{3} & \multirow{4}{*}{4} & \multirow{4}{*}{exterior} & 0.080, 0.327 & 3.79\ee{-3} & --  & 3.44\ee{-4}  & -- \\
                     &                      &                       & 0.040, 0.170 & 1.88\ee{-4} & 4.3 & 9.34\ee{-6} & 5.2 \\
                     &                      &                       & 0.020, 0.085 & 1.30\ee{-5} & 3.9 & 2.43\ee{-7} & 5.3 \\
                     &                      &                       & 0.010, 0.043 & 5.24\ee{-7} & 4.6 & 7.44\ee{-9} & 5.0 \\
\bottomrule
\end{tabular}
\vspace*{0.2in}
\begin{tabular}{c c}
    \resizebox{1.8in}{!}{\begin{tikzpicture}
   \path [fill=white](-1.2,-1) rectangle (2.3,1.4);
    \draw [fill=\intHatColor!40,thick]( - 0.8, - 0.8) rectangle (0.8,  0.8);
    \draw[pattern=\intPatternType, thick, pattern color=\intPatternColor] plot file{starfish-curve.dat} -- cycle;
     \node at (0, -1.0) {\scriptsize interior};
     \end{tikzpicture}}    &           \resizebox{1.8in}{!} {\begin{tikzpicture}
     \path[fill=white](-1.2,-1) rectangle(2.3, 1.4);
       \draw [fill=\extHatColor!60, thick](-1,-1) rectangle (1,1);
       \path [pattern=\intPatternType, pattern color=\intPatternColor](-1,-1) rectangle (1,1);
     \draw [fill=\extHatColor!60, thick] (0., 0.) plot file{starfish-curve.dat} -- cycle;
    \node at (1.7, -0.8) {\scriptsize exterior};
    \end{tikzpicture}}  \\
  \bottomrule
\end{tabular}
\end{center}
\end{table}

\subsection{Numerical considerations}
We consider two approaches for solving the coupled four-variable linear
system: solving the full system together, or implementing the
Schur complement based procedure described
in~\eqref{eq:blockFEsolve}--\eqref{eq:interface-split-solve-IE}.
In the following, the total $4\times 4$ system was preconditioned with
a block preconditioner on the FE blocks; the preconditioner used was
smoothed aggregation~AMG from \texttt{pyamg}~\cite{pyamg-github}.  For
the Schur complement solve, the action of inverting the FE operators was implemented
by separating the Dirichlet nodes to regain a symmetric positive
definite system for the interior points.  This inner solve then
used preconditioned CG\@.  The total computational work in the Schur
complement
solve depends on both the number of outer GMRES iterations for the
densities and inner iterations on the FE solutions, carried out every time
the action of $\tensor{\FEsolve}$ is required as part of the outer
operator.  Considering the structure of the FE block of the coupled operator $\mathcal{C}$
from \eqref{eq:generalmat-split}
and the definition of the solution operator $\tensor\FEsolve$ in \eqref{eq:blockFEsolve},
the inner and exterior FE problems are only coupled through the total system \eqref{eq:interface-split-solve-IE}; we
choose to solve the interior and exterior FE problems separately in each application of $\tensor\FEsolve$.

\ifpaper{
  To illustrate the effectiveness of the described solution procedure
  and the effectiveness of our preconditioning, we show the number of
  required inner and outer iterations for the
  quadratic-log test case in Figure~\ref{fig:quad-log-split-its}.  Due
  to the well-conditioned nature of the second-kind integral equation
  discretizations, the number of outer iterations remains nearly
  constant as the mesh spacing $\hie$ decreases.  The total number of interior and
  exterior inner iterations increases with both decreasing mesh size and
  increasing order of basis functions, which is consistent with the
  expected behavior of standard finite element discretizations.
    \begin{center}
    \begin{figure}[!ht]
    \begin{center}
    \includegraphics[width=0.68\textwidth]{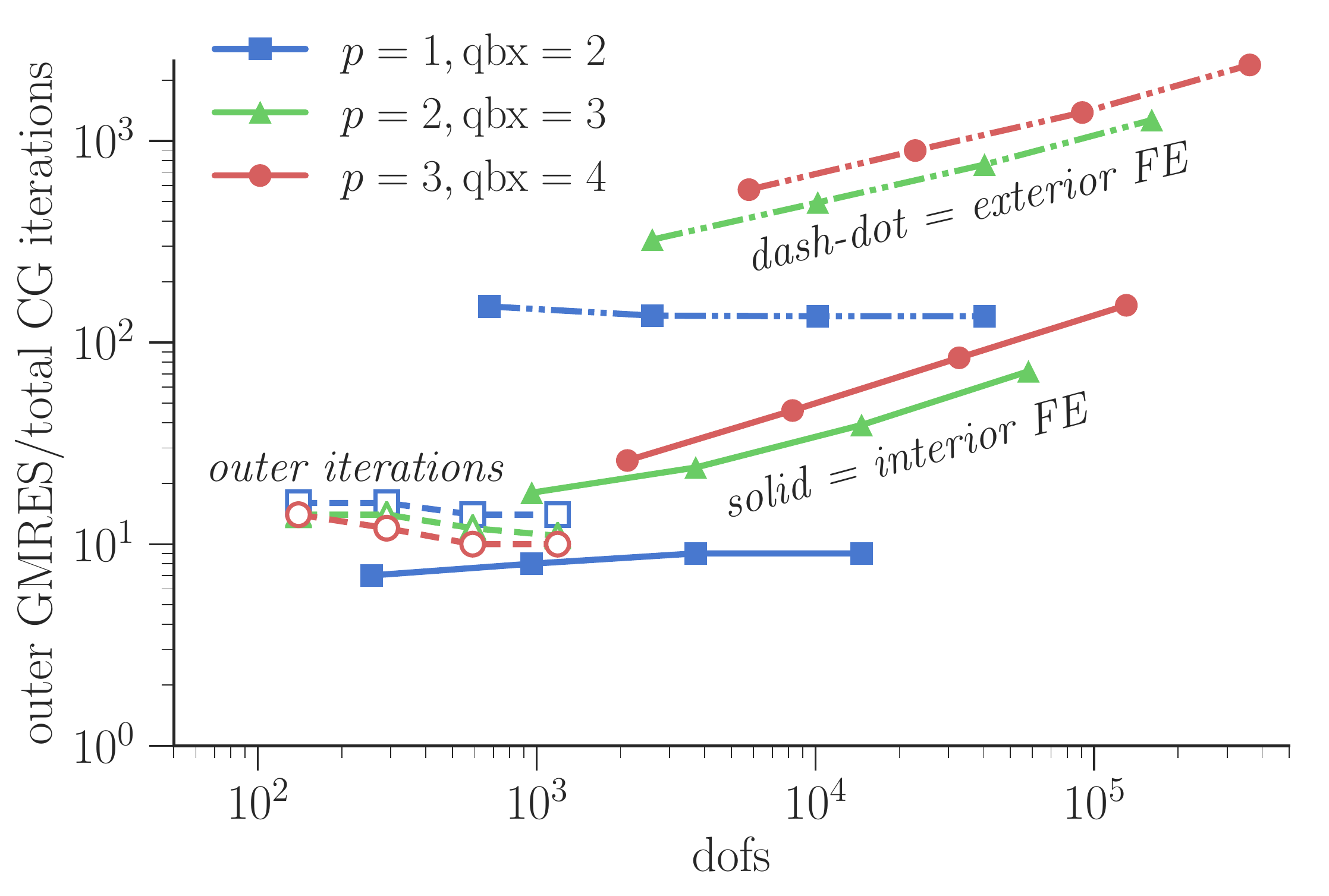}
    \caption{Outer GMRES(20) iterations to solve \eqref{eq:interface-split-solve-IE} and
    inner preconditioned CG iterations of interior/exterior FE problems
     for the quadratic-log problem.}
    \label{fig:quad-log-split-its}
    \end{center}
    \end{figure}
  \end{center}
}

  \section{Conclusions}
We have demonstrated a method of coupling finite-element and integral equation solvers for the
strong enforcement of boundary conditions on interior and exterior embedded boundaries.  Furthermore, we have
introduced a new method of coupling these FE-IE subproblems to solve a wide class of interface problems with
homogeneous and non-homogeneous jump conditions.
Our method does not require any special modifications to the finite element
basis functions.  It also does not need volume mesh refinement around the embedded domain,
which means that time-dependent domains will not necessitate modifications to the volume-based finite element
matrices---only the surface-based integral equation discretizations and the coupling matrices would be updated.
One benefit is that our method can be implemented with off-the-shelf FE and IE packages.  We have shown theoretical error bounds in the
case of the interior embedded mesh problem and achieved empirical high-order convergence for interior,
exterior, and interface examples, even very close to the embedded boundary.

\section*{Acknowledgments}

Portions of this work were sponsored by the Air Force Office of
Scientific Research under grant number FA9550–12–1–0478, and by the
National Science Foundation under grant number CCF-1524433.
Part of the work was performed while NB and AK were participating in
the HKUST-ICERM workshop `Integral Equation Methods, Fast
Algorithms and Their Applications to Fluid Dynamics and Materials
Science' held in 2017.

\bibliography{FErefs}

\end{document}